\documentclass[12pt]{article}   
   
\usepackage{latexsym,amsfonts,amsmath,mathrsfs,epsfig,tabularx,amsthm,amssymb,enumitem,bm}
\usepackage{float}
\usepackage{mathtools}

\usepackage[title,titletoc]{appendix}

\usepackage{microtype}
   
\usepackage[draft=false,colorlinks,bookmarksnumbered,linkcolor=black, citecolor=black]{hyperref}

\usepackage{aliascnt}

\newtheorem{theorem}{Theorem}[section]

\newaliascnt{lem}{theorem}
\newtheorem{lemma}[lem]{Lemma}

\aliascntresetthe{lem}

\newaliascnt{ass}{theorem}

\aliascntresetthe{ass}

\newaliascnt{prop}{theorem}
\newtheorem{prop}[prop]{Proposition}

\aliascntresetthe{prop}

\newaliascnt{cor}{theorem}
\newtheorem{corollary}[cor]{Corollary}

\aliascntresetthe{cor}

\newaliascnt{defi}{theorem}

\aliascntresetthe{defi}

\theoremstyle{definition}
\newaliascnt{ex}{theorem}

\aliascntresetthe{ex}

\newaliascnt{rem}{theorem}
\newtheorem{remark}[rem]{Remark}

\aliascntresetthe{rem}

\renewcommand{\d}{\,\mathrm{d}}											
\renewcommand*{\epsilon}{\varepsilon}                                   
\renewcommand*{\rho}{\varrho}                                   		
\newcommand*{\sep}{\; \vrule \;}                                        
\newcommand*{\N}{\mathbb{N}}                                            
\newcommand*{\R}{\mathbb{R}}                                            
\newcommand*{\C}{\mathbb{C}}                                            
\renewcommand*{\S}{\mathcal{S}}                                         
\newcommand*{\D}{\mathcal{D}}                                         	
\newcommand*{\J}{\mathcal{J}}                                         	

\newcommand*{\abs}[1]{\left| #1 \right|}                                
\newcommand*{\norm}[1]{\left\| #1 \right\|}                             
\newcommand*{\ceil}[1]{\left\lceil #1 \right\rceil}                     
\newcommand*{\link}[1]{(\ref{#1})}                                      

\newcommand*{\distr}[2]{\left\langle #1, #2 \right\rangle}              

\renewcommand{\tilde}[1]{ \widetilde{#1} }        						
\DeclareMathOperator{\supp}{supp}										
\newcommand{\esssup}{ \mathop{\mathrm{esssup}}\limits }

\usepackage{footmisc}

\usepackage{tikz}
\usetikzlibrary{arrows,decorations.pathmorphing,decorations.pathreplacing,backgrounds,positioning,fit,petri}

\usepackage[margin=0.4cm, format=hang]{caption}

\setlist{itemsep=2pt, topsep=2pt}

\title{On the lack of interior regularity of the\\ $p$-Poisson problem with $p>2$}

\author{Markus Weimar\footnote{Ruhr University Bochum, Faculty of Mathematics, Universit\"atsstra{\ss}e 150, 44801 Bochum, Germany. Email: markus.weimar@rub.de }} 

\begin{document}   
\maketitle   

\begin{abstract}
\noindent In this note we are concerned with interior regularity properties of the $p$-Poisson problem $\Delta_p(u)=f$ with $p>2$. For all $0<\lambda\leq 1$ we constuct right-hand sides $f$ of differentiability $-1+\lambda$ such that the (Besov-) smoothness of corresponding solutions $u$ is essentially limited to $1+\lambda / (p-1)$. The statements are of local nature and cover all integrability parameters. They particularly imply the optimality of a shift theorem due to Savar{\'e} [J.\ Funct.\ Anal.\ 152:176--201, 1998], as well as of some recent Besov regularity results of Dahlke et al.\ [Nonlinear Anal.\ 130:298--329, 2016].


\end{abstract}

\section{Introduction and main results}
In what follows we deal with interior regularity properties of solutions $u\in W^1_p(\Omega)$ to the $p$-Poisson problem
\begin{align}\label{eq:strong_ppoisson}
	-\Delta_p(u)=f
\end{align}
on bounded Lipschitz domains $\Omega \subset\R^d$ for $d\in\N$ and $1<p<\infty$. Here the $p$-Laplace operator $\Delta_p$ is given by
$$
	\Delta_p(u):=\mathrm{div}(A(\nabla u)),
	\qquad \text{where} \qquad A(\nabla u):=\abs{\nabla u}^{p-2}\nabla u
$$ 
is called the naturally associated vector field. For distributions $f\in W^{-1}_{p'}(\Omega):=(W^1_{p,0}(\Omega))'$, with $1/p+1/p'=1$, the corresponding variational formulation is given by
\begin{align*}
	\int_\Omega \distr{A(\nabla u)(x)}{\nabla \psi(x)}_{\R^d} \d x = f(\psi), \qquad \psi \in \D(\Omega),
\end{align*}
where $\D(\Omega):=C_0^\infty(\Omega)$ is the set of test functions on $\Omega$, the space $W^1_{p,0}(\Omega)$ is the closure of $\D(\Omega)$ w.r.t.\ the first order $L_p$-Sobolev norm, and $\distr{\cdot}{\cdot}_{\R^d}$ denotes the 
inner product on $\R^d$.

Equations of type \link{eq:strong_ppoisson} arise in various applications such as non-Newtonian fluid theory, rheology, radiation of heat and many others. In fact the quasi-linear operator $\Delta_p$ has a similar model character for nonlinear problems as the ordinary Laplacian (i.e., the case $p=2$) for linear problems. Meanwhile, many results concerning existence and uniqueness of solutions are known. For details we refer to \cite{Lin2006} and the references therein. However, most of these results deal with classical function spaces of H\"older or Sobolev type. On the other hand, in view of strong relations to nonlinear approximation classes and adaptive numerical algorithms, regularity results in more general smoothness spaces of Besov type became more and more important in recent times; see, e.g., \cite{CDD01, GasMor2014}. For the $p$-Poisson equation \link{eq:strong_ppoisson} and related problems only few results are known in this direction, see \cite{BalDieWei2018, DahDieHar+2016, HarWei2018}, as well as \cite{DeT1982, Sav1998, Sim1978}.

Let us recall that for $0<\rho,q\leq\infty$ and $s\in\R$ Besov spaces $B^s_{\rho,q}(\R^d)$ are quasi-Banach spaces which can be defined as subsets of tempered distributions $f\in\S'(\R^d)$ by means of harmonic analysis. The corresponding spaces on domains (i.e.\ connected open subsets of $\R^d$) are then defined by restriction such that we obtain subsets of $\D'(\Omega)$. 

\begin{remark}[Function spaces]
We assume that the reader is familiar with the basics of function space theory as it can be found, e.g., in the monographic series of Triebel \cite{T83,T92,T06,T08}. Anyhow, let us mention that by now various equivalent characterizations, embeddings, interpolation and duality assertions for the scale of Besov spaces are known. Without going into further details, let us recall the following results, valid for bounded Lipschitz domains $\Omega\subset\R^d$ or $\Omega=\R^d$ itself:
\begin{itemize}
\item[(i)] For $0<\rho,q\leq \infty$, and $s > d \, \max\{0,1/\rho-1\}$ the Besov space $B^s_{\rho,q}(\Omega)$ can be characterized as collection of all $g\in L_\rho(\Omega)$ for which 
\begin{equation*}
	\abs{g}_{B^s_{\rho,q}(\Omega)}
	:=\begin{cases}
		\displaystyle\bigg( \int_0^{\overline{t}} \bigg[ t^{-s} \sup_{\substack{h\in\R^d,\\\abs{h}_2\leq t}} \norm{\Delta_h^k g \sep L_{\rho}(\Omega_{h,k})} \bigg]^q \frac{\d t}{t} \bigg)^{1/q}, &0<q<\infty,\\
		\displaystyle\sup_{0<t\leq \overline{t}} t^{-s} \sup_{\substack{h\in\R^d,\\\abs{h}_2\leq t}} \norm{\Delta_h^k g \sep L_{\rho}(\Omega_{h,k})}, &q=\infty,
	\end{cases}
\end{equation*}
is finite \cite[Sect.\ 1.11.9]{T06}. 

In fact, the expression 
$$
	\norm{g \sep B^s_{\rho,q}(\Omega)} := \norm{g \sep L_\rho(\Omega)}+\abs{g}_{B^s_{\rho,q}(\Omega)}
$$ 
provides a quasi-norm on $B^s_{\rho,q}(\Omega)$. Here we assume $\overline{t}>0$ and $k> s$ to be fixed. Further, 
$\Delta_h^k g$ denotes the $k$-th order finite difference of $g$ with step size $h\in\R^d$ and  
$$
	\Omega_{h,k}:=\{x \in \R^d \sep x+\ell h \in \Omega \,\text{ for all }\, \ell=0,\ldots,k \}.
$$ 
\item[(ii)] For $0<s\notin \N$ we have $B_{\infty,\infty}^s(\Omega)=C^s(\Omega)$ (H\"older spaces) and
$$
	W^s_\rho(\Omega)=B^s_{\rho,\rho}(\Omega) \quad \text{for all} \quad 1 \leq \rho<\infty, 
$$
(Sobolev-Slobodeckij spaces) in the sense of equivalent norms. 
\end{itemize}
So, roughly speaking, in $B^s_{\rho,q}(\Omega)$ we collect all $g$ such that their weak partial derivatives $D^\alpha g$ up to order $s$ belong to the Lebesgue space $L_{\rho}(\Omega)$. The third parameter $0<q\leq \infty$ acts as a minor important fine index.
\hfill$\square$
\end{remark}

In his seminal paper \cite{Sav1998} Savar{\'e} developed a variational argument which allows to show the following shift theorem for the $p$-Poisson problem:
\begin{prop}[{see \cite[Thm.\ 2]{Sav1998}}]\label{prop:savare}
	For $d\in\N$ let $\Omega\subset\R^n$ be a bounded Lipschitz domain. For $2<p<\infty$ and $f \in W^{-1}_{p'}(\Omega)$ let $u$ be the unique weak solution to \link{eq:strong_ppoisson} in $W_{p,0}^1(\Omega)$.
	Then for all $\lambda \in (0,1/p')$ the following implications hold:
	\begin{align}\label{eq:shift}
		f \in W_{p'}^{-1+\lambda}(\Omega) \quad \Longrightarrow \quad u \in W_{p}^{1+\frac{\lambda}{p-1}}(\Omega)
	\end{align}
	and
	$$
		f \in B_{p',1}^{-1+ \frac{1}{p'}}(\Omega) \quad \Longrightarrow \quad u \in B_{p,\infty}^{1+\frac{1/p'}{p-1}}(\Omega).
	$$
\end{prop}
In addition, Savar{\'e} claims that \link{eq:shift} is ``optimal'' \cite[Rem.\ 4.3]{Sav1998} and refers to Simon \cite{Sim1978}. But Simon's optimality results refer to a (slightly) different equation on the whole of $\R^d$ with right hand sides in $L_{r}$ or $C^\infty$ and hence they do not cover  Savar{\'e}'s claim at all. 

However, it is possible to use similar ideas in order to show the following \autoref{thm:main} which constitutes the main result of this note. It states that for $p>2$ and all $0<\lambda<1$ there are right-hand sides of smoothness $-1+\lambda$ such that the smoothness of corresponding solutions to the $p$-Poisson problem \link{eq:strong_ppoisson} is essentially limited by $1+\lambda/(p-1)$. Moreover, this actually holds \emph{independently} of the integrability parameter. That is, in sharp contrast to point singularities, we \emph{do not gain smoothness} when derivatives are measured in weaker $L_\rho$-norms.
For a (constructive) proof we refer to \autoref{subsect:proof-main} below.

\begin{theorem}\label{thm:main}	
	Let $d\in \N$ and $2 \leq p<\infty$ be fixed. Further let $\Omega \subseteq \R^d$ be either $\R^d$ itself, a bounded Lipschitz domain, or an interval (if $d=1$). Moreover, assume $0<\epsilon<1/p$. 
	Then for all $\epsilon\,(p-1)<\lambda < 1-\epsilon$ and $1<\mu \leq \infty$ there exists a right-hand side
	$$
		f=f_{\lambda,\mu}\in B^{-1+\lambda}_{\mu,\infty}(\Omega)\cap W^{-1}_{p'}(\Omega)
	$$ 
	with compact support in $\Omega$ 
	such that the corresponding weak solution $u \in W^1_{p,0}(\Omega)$ to \link{eq:strong_ppoisson} is compactly supported as well
	and satisfies
	\begin{align}\label{eq:reg_u}
	  	u \in \left\{\begin{array}{rclcrlcr}
			\!\! B^{1+\frac{\lambda}{p-1}-\epsilon}_{\rho,q}(\Omega) \!\!\!&\setminus&\!\!\! B^{1+\frac{\lambda}{p-1}}_{\rho,q}(\Omega) 
				&\;\text{if}&  \mu(p-1) < \rho \!\!\!&\leq \infty &\text{and}& 0<q \leq \infty,\\
			\!\! B^{1+\frac{\lambda}{p-1}}_{\rho,\infty}(\Omega)  \!\!\!&\setminus&\!\!\! B^{1+\frac{\lambda}{p-1}}_{\rho,q}(\Omega)
			 	&\;\text{if}&  \rho \!\!\!&= \mu(p-1) &\text{and}& 0<q<\infty,\\
			\!\! B^{1+\frac{\lambda}{p-1}}_{\rho,q}(\Omega)  \!\!\!&\setminus&\!\!\! B^{1+\frac{\lambda}{p-1}+\epsilon}_{\rho,q}(\Omega) 
				&\;\text{if}&  1 \leq \rho \!\!\!&< \mu(p-1) &\text{and}& 0<q\leq\infty.
		\end{array}\right. 
	\end{align}
	In addition, the naturally associated vector field satisfies
	\begin{align}\label{eq:reg_A}
		A(\nabla u) \in \left\{\begin{array}{rclcrlcr}
			\big( B^{\lambda-\epsilon}_{\rho,q}(\Omega) \!\!\!&\setminus&\!\!\! B^{\lambda}_{\rho,q}(\Omega) \big)^d 	
			 	&\quad\text{if}&  \mu < \rho \!\!\!&\leq \infty &\text{and}& 0<q \leq \infty,\\
			\big( B^{\lambda}_{\rho,\infty}(\Omega) \!\!\!&\setminus&\!\!\! B^{\lambda}_{\rho,q}(\Omega) \big)^d
				&\quad\text{if}&  \rho \!\!\!&= \mu &\text{and}& 0<q<\infty,\\
			\big( B^{\lambda}_{\rho,q}(\Omega) \!\!\!&\setminus&\!\!\! B^{\lambda+\epsilon}_{\rho,q}(\Omega) \big)^d 
				&\quad\text{if}&  1 \leq \rho \!\!\!&< \mu &\text{and}& 0<q\leq\infty.
		\end{array}\right. 
	\end{align}
\end{theorem}

Before we proceed some general comments are in order:
\begin{remark}\label{rem:general}
First of all, let us stress the point that, due to the compact support of $f$ and~$u$, \autoref{thm:main} is of \emph{local} nature.
 
Secondly, we note that the restriction to $\rho\geq 1$ is for notational convinience only. Using standard embeddings (see \autoref{lem:embedding}\link{item:smallerrho} below) and complex interpolation (see, e.g., Kalton \emph{et~al.} \cite[Theorem~5.2]{KMM07}) we can easily extend \link{eq:reg_u} by 
$$
	u \in B^{1+\frac{\lambda}{p-1}}_{\rho,q}(\Omega)  \;\setminus\; B^{1+\frac{\lambda}{p-1}+c_\rho\epsilon}_{\rho,q}(\Omega)
	\qquad \text{for all} \qquad 0<\rho<1 \qquad \text{and} \qquad  0<q\leq\infty
$$ 
with some $c_\rho \sim 1/\rho$. Likewise, the same arguments can be used to extend also \link{eq:reg_A}.
\hfill$\square$
\end{remark}

Observe that \autoref{thm:main} applied for $\mu:=p'$ indeed shows optimality of Savar{\'e}'s result in some sense:

\begin{corollary}
	In \autoref{prop:savare} the smoothness of $u$ w.r.t.\ $L_p$ cannot be improved without strengthening the assumptions on~$f$. 
\end{corollary}
\begin{proof}
	 Choosing $\lambda:=\tilde{\lambda}+(p-1)\delta$ with some $\tilde{\lambda}\in(0,1/p')$ and $\delta>0$ arbitrarily small, \autoref{thm:main} allows to find a right-hand side $f\in B_{\mu,\infty}^{-1+\lambda}(\Omega) \hookrightarrow B_{p',p'}^{-1+\tilde{\lambda}}(\Omega)=W_{p'}^{-1+\tilde{\lambda}}(\Omega)$ such that the corresponding solution of the Dirichlet problem for the $p$-Poisson equation satisfies $u \notin B_{p,p}^{1+\frac{\tilde{\lambda}}{p-1}+\delta}(\Omega)=W_{p}^{1+\frac{\tilde{\lambda}}{p-1}+\delta}(\Omega)$. 
	Similarly for $\lambda:=1/p'+\delta$ with $\delta>0$, there exists $f\in B_{\mu,\infty}^{-1+\lambda}(\Omega) \hookrightarrow B_{p',1}^{-1+\frac{1}{p'}}(\Omega)$ such that $u \notin B_{p,q}^{1+\frac{\lambda}{p-1}}(\Omega) \hookleftarrow B_{p,\infty}^{1+\frac{1/p'}{p-1}+\delta}(\Omega)$. In view of \autoref{rem:general}, these examples remain valid also on smooth domains.
\end{proof}

Furthermore, \autoref{thm:main} shows that regarding regularity questions it seems better to look at the the mapping $f\mapsto A(\nabla u)$, rather than $f\mapsto u$. In fact, in view of the case $\rho=\mu$ in~\link{eq:reg_A}, one might conjecture the existence of a $p$-independent mechanism which (for some range of parameters) locally transfers exactly one order of regularity from the right-hand side~$f$ to the naturally associated vector filed $A(\nabla u)$. For the case $d=2$ this already has been verified in \cite{BalDieWei2018}. In other words, \autoref{thm:main} shows that also their results cannot be improved.

\autoref{thm:main} is complemented by
\begin{theorem}\label{thm:mainL}	
	Let $d\in \N$ and $2 < p<\infty$. Further let $\Omega \subseteq \R^d$ be either $\R^d$ itself, a bounded Lipschitz domain, or an interval (if $d=1$). Moreover, let $0<\epsilon<\min\{1/(p-1), 1-1/(p-1)\}$. 
	Then for all $1<\mu \leq \infty$ there exists a compactly supported right-hand side 
	$$
		f=f_\mu \in L_\nu(\Omega)\cap W^{-1}_{p'}(\Omega) \quad \text{for all} \quad 
		\begin{cases}
			0<\nu < \mu & \quad \text{if} \quad \mu <\infty, \\
			0<\nu \leq \infty & \quad \text{if} \quad \mu =\infty,
		\end{cases}
	$$
	such that the corresponding weak solution $u \in W^1_{p,0}(\Omega)$ to \link{eq:strong_ppoisson} is compactly supported as well and satisfies \link{eq:reg_u} with $\lambda=1$. Moreover, then for $1<\rho<\infty$ there holds
	$$
		A(\nabla u) \in \big( W^1_\rho(\Omega) \big)^d \qquad \text{if and only if} \qquad	\rho < \mu.
	$$
\end{theorem}

Here \autoref{rem:general} applies likewise. Moreover, also this result implies certain optimality statements:

\begin{remark}
At first, setting $\mu:=\rho:=\infty$ in \autoref{thm:mainL}, we recover the well-known assertion that for bounded right-hand sides the local H\"older regularity of the gradient $\nabla u$ of solutions to the $p$-Poisson equation \link{eq:strong_ppoisson} with $p>2$ is bounded by $1/(p-1)$.

Secondly, in 
\cite{DahDieHar+2016} it has been shown that for $p>2$, bounded Lipschitz domains $\Omega\subset\R^2$, and right-hand sides $f\in L_\infty(\Omega)$ the unique solution $u\in W^1_{p,0}(\Omega)$ to \link{eq:strong_ppoisson} satisfies
$$
	u\in B^\sigma_{\tau_\sigma,\tau_\sigma}(\Omega) \qquad \text{for all} \qquad 0<\sigma < \overline{\sigma} := 1+\frac{1}{p-1} \qquad \text{and} \qquad \frac{1}{\tau_\sigma}=\frac{\sigma}{2}+\frac{1}{p}. 
$$
In view of \autoref{thm:mainL} and \autoref{rem:general}, $\overline{\sigma}$ cannot be replaced by any larger number. 
\hfill$\square$
\end{remark}

The rest of this note is devoted to the proofs of \autoref{thm:main} and \autoref{thm:mainL}, respectively. \autoref{subsect:prep} collects some quite technical preparations. Afterwards, the statements are proven easily in \autoref{subsect:proof-main} and \autoref{subsect:proof-mainL}. 

\textbf{Notations:}
In the sequel $\N$ denotes the natural numbers without zero and we use $\R_+$ for the set of strictly positive reals. For families $\{a_{j} \sep j\in \J\}$ and $\{b_{j} \sep j\in \J\}$ of non-negative reals over a common index set $\J$ we write $a_{j} \lesssim b_{j}$ if there exists a constant $c>0$ (independent of the context-dependent parameters $j$) such that
$	
	a_{j} \leq c\cdot b_{j}
$
holds uniformly in $j\in\J$.
Consequently, $a_{j} \sim b_{j}$ means $a_{j} \lesssim b_{j}$ and $b_{j} \lesssim a_{j}$. In addition, the symbol $\hookrightarrow$ is used to denote continuous embeddings. 

\section{Proofs}\label{sect:proof}
Our main proofs given in \autoref{subsect:proof-main} and \autoref{subsect:proof-mainL} below require some preparations. The basic idea will be based on a construction given by Simon \cite[Sect.\ 4]{Sim1979}.
\subsection{Preparations}\label{subsect:prep}
For $1<\theta <\infty$ define the sequence
$$
a_{n,\theta}:=4 \sum_{j=1}^{n-1} j^{-\theta}, \qquad n\in\N \setminus\{1\}.
$$
Then for all $n\geq 3$
$$
4=a_{2,\theta} < \ldots < a_{n,\theta}< a_{n+1,\theta}=a_{n,\theta}+4 \,n^{-\theta} < \ldots < a_{\infty,\theta}:=\lim_{n\to\infty} a_{n,\theta}=4\,\zeta(\theta)<\infty.
$$
Further, with $\sigma \in \R_+$ let $w_{\sigma,\theta} \colon \R \to [0,\infty)$ be defined piecewise by
$$
	w_{\sigma,\theta}(\xi) := \left\{\begin{array}{lcrl}
	(\xi-a_{n,\theta})^\sigma 
		&\quad\text{if}&  a_{n,\theta} \leq \xi \!\!\!&< a_{n,\theta} + n^{-\theta}, \\
	n^{-\theta\sigma}  
		&\quad\text{if}&  a_{n,\theta} + n^{-\theta} \leq \xi \!\!\!&< a_{n,\theta} + 2\,n^{-\theta}, \\
	(a_{n,\theta} + 3\, n^{-\theta}-\xi)^\sigma  
		&\quad\text{if}&  a_{n,\theta} + 2\,n^{-\theta} \leq \xi \!\!\!&< a_{n,\theta} + 3\,n^{-\theta}, \\
	0  
		&\quad\text{if}&  a_{n,\theta} + 3\,n^{-\theta} \leq \xi \!\!\!&< a_{n+1,\theta}, 
	\end{array}\right. 
$$
on $[a_{n,\theta}, a_{n+1,\theta})$, $n\geq 2$, and
$$
w_{\sigma,\theta}(\xi):=0 \quad \text{on} \quad \R\setminus [a_{2,\theta},a_{\infty,\theta}).
$$
Moreover, let us define $\mathscr{S}_\theta:=[4,4\,\zeta(\theta)]\subset\R$, as well as the set of transition points
$$
\mathscr{P}_{\theta}:=\big\{\xi \in [a_{2,\theta},a_{\infty,\theta}) \sep \xi=a_{n,\theta}+k\,n^{-\theta} \text{ for some } n\geq 2 \text{ and } k\in\{0,\ldots,3\} \big\} \cup \{a_{\infty,\theta}\} \subset \R.
$$

\begin{lemma}[Properties of $w_{\sigma,\theta}$]\label{lem:prop_w}
Let $\sigma\in \R_+$ and $1<\theta <\infty$, as well as $0<\rho\leq \infty$. Then
\begin{enumerate}[label=(\roman*),ref=\roman*]
	\item \label{item:supp_w} $w_{\sigma,\theta}$ is continuous with compact support
	$\supp (w_{\sigma,\theta}) \subseteq \mathscr{S}_\theta$. 	
	\item \label{item:countableT} $\mathscr{P}_\theta$ is countable and $w_{\sigma,\theta}$ is continuously differentiable on $\R\setminus \mathscr{P}_\theta$, i.e., $w_{\sigma,\theta}'$ exists a.e.
	\item \label{item:wgammasigma} $w_{\gamma\sigma,\theta}=(w_{\sigma,\theta})^{\gamma-1} w_{\sigma,\theta}$ for all $\gamma \in\R_+$.
	\item \label{item:bound_w} $0\leq w_{\sigma,\theta}(\xi) \leq 2^{-\sigma\theta}$ for all $\xi\in \R$.
\end{enumerate}
\end{lemma}

\begin{proof}
All statements are obvious consequences of the definition of $w_{\sigma,\theta}$. In order to
see \link{item:bound_w}, note that $0\leq w_{\sigma,\theta}(\xi) \leq n^{-\sigma\theta}$ on $[a_{n,\theta}, a_{n+1,\theta}]$ with $n\geq 2$.
\end{proof}

In the sequel, we will need sharp regularity assertions for $w_{\sigma,\theta}$. Before we state and prove them, let us recall some well-known embedding results for Besov spaces which are proven here for the sake of completeness.

\begin{prop}\label{lem:embedding}
For $d\in \N$ let $\Omega \subseteq \R^d$ be either $\R^d$ itself, a bounded Lipschitz domain, or an interval (if $d=1$). Further let $0<\rho,\rho_0,\rho_1,q,q_0,q_1\leq \infty$ and $s,s_0,s_1\in\R$. Then
\begin{enumerate}[label=(\roman*),ref=\roman*]
\item \label{item:lift} $B^{1+s}_{\rho,q}(\Omega)=\big\{g \in B^{s}_{\rho,q}(\Omega) \sep \nabla g \in (B^{s}_{\rho,q}(\Omega))^d\big\}$ with 
$$
	\norm{g \sep B^{1+s}_{\rho,q}(\Omega)} \sim 
	\sum_{\substack{\alpha\in\N_0^d:\\ \abs{\alpha}_1\leq 1}} \norm{D^\alpha g \sep B^{s}_{\rho,q}(\Omega)}.
$$
\item \label{item:smallers} If $s_1<s_0$, then $g\in B^{s_0}_{\rho,q_0}(\Omega)$ implies $g\in B^{s_1}_{\rho,q_1}(\Omega)$.
\item \label{item:smallerrho} If $\rho_1\leq \rho_0$ and $g\in B^{s}_{\rho_0,q}(\Omega)$ has compact support in $\R^d$, then $g\in B^{s}_{\rho_1,q}(\Omega)$.
\end{enumerate}
\end{prop}
\begin{proof}
	Assertion \link{item:lift} is a special instance of \cite[Prop.\ 4.21]{T08}. So, let us prove \link{item:smallers} and \link{item:smallerrho}.
	By means of Rychkov's extension operator \cite{Ryc1999} we can w.l.o.g.\ assume that $\Omega=\R^d$. Further, let $c(g)$ denote the sequence of wavelet coefficients of $g$ w.r.t.\ a sufficiently smooth Daubechies wavelet system on $\R^d$. Then the wavelet isomorphism from \cite[Thm.\ 3.5]{T06} implies that $\norm{g \sep B^{s}_{\rho,q}(\R^d)} \sim \norm{c(g) \sep b^{s}_{\rho,q}(\nabla)}$ for all $0<\rho,q\leq \infty$ and $s\in\R$ with $b^{s}_{\rho,q}(\nabla)$ being suitable sequence spaces. 
Now \link{item:smallers} follows from the standard embedding $b^{s_0}_{\rho_0,q_0}(\nabla) \hookrightarrow b^{s_1}_{\rho_0,q_1}(\nabla)$ if $s_1<s_0$ which can be found, e.g., in \cite[Prop.\ 2.5]{Wei2016}. 

In order to prove \link{item:smallerrho}, we note that the compact support of $g$ implies that $\norm{c(g) \sep b^{s}_{p,q}(\nabla)}$ equals $\norm{c(g) \big|_{\widetilde{\nabla}} \sep b^{s}_{p,q}(\widetilde{\nabla})}$, where $b^{s}_{p,q}(\widetilde{\nabla})$ refers to the corresponding sequence space for some bounded domain. 
For these spaces there holds the embedding $b^{s_0}_{\rho_0,q_0}(\widetilde{\nabla}) \hookrightarrow b^{s_0}_{\rho_1,q_0}(\widetilde{\nabla})$ if $\rho_1\leq \rho_0$, see \cite[Prop.~2.5]{Wei2016} again.
\end{proof}
\newpage
\begin{lemma}[Regularity of $w_{\sigma,\theta}$]\label{lem:Regw} 
	Let $\sigma\in \R_+$ and $1<\theta <\infty$, as well as $0<\rho\leq \infty$. Then
	\begin{enumerate}[label=(\roman*),ref=\roman*]
		\item \label{item:w_in_L_rho} $w_{\sigma,\theta} \in L_{\rho}(\R)$,
		\item \label{item:w'_in_L_rho} $w_{\sigma,\theta}'\in L_\rho(\R)$ holds if and only if
			\begin{align}\label{cond:w'_in_L_rho}
				\sigma \geq 1,
				\qquad\qquad \text{or} \qquad\qquad 
				0<\sigma < 1 \quad \text{and} \quad \frac{1-\sigma}{1-1/\theta} < \frac{1}{\rho}.
			\end{align}
		\item \label{item:w_in_B} Additionally assume $0<\sigma<1/\theta<1$ and 
			\begin{align}\label{cond:rho}
				0 \leq \frac{1}{\rho} < \min\left\{ \theta\,(1+\sigma), \frac{1-\sigma}{1-1/\theta}\right\}.
			\end{align}
			Then $w_{\sigma,\theta} \in B^{s}_{\rho,q}(\R)$ holds if and only if
			$$
				s=\sigma+\frac{1-1/\theta}{\rho} \quad\text{and}\quad q=\infty, 
				\quad\qquad \text{or} \qquad\quad 
				s<\sigma+\frac{1-1/\theta}{\rho} \quad\text{and}\quad 0<q \leq\infty.
			$$	
	\end{enumerate}
\end{lemma}
\begin{proof}
	In view of \autoref{lem:prop_w}\link{item:supp_w} assertion \link{item:w_in_L_rho} is obvious.

	Let us show \link{item:w'_in_L_rho}. 
Clearly $w_{\sigma,\theta}'\in L_\infty(\R)$ is equivalent to $\sigma \geq 1$. On the other hand, for $0<\rho<\infty$ we have
\begin{align*}
	\norm{w_{\sigma,\theta}'\sep L_\rho(\R)}^\rho
	&= \sum_{n=2}^\infty \left( \int_{a_{n,\theta}}^{a_{n,\theta}+n^{-\theta}} \abs{w_{\sigma,\theta}'(\xi)}^\rho \d\xi + \int_{a_{n,\theta}+2\,n^{-\theta}}^{a_{n,\theta}+3\,n^{-\theta}} \abs{w_{\sigma,\theta}'(\xi)}^\rho \d\xi \right) \\
	&= 2\,\sigma^\rho \, \sum_{n=2}^\infty \int_{0}^{n^{-\theta}} x^{(\sigma-1)\rho} \d x.
\end{align*}
The latter integral is finite if and only if $1+(\sigma-1)\rho > 0$. In this case there holds
\begin{align*}
	\norm{w_{\sigma,\theta}'\sep L_\rho(\R)}^\rho
	&\sim \sum_{n=2}^\infty n^{-\theta(1+(\sigma-1)\rho)} 
	= \zeta\big( \theta(1+(\sigma-1)\rho) \big)-1
\end{align*}
which is finite if only if the argument of the Riemann zeta function $\zeta$ is strictly larger than one.
Thus, for $0<\rho<\infty$ we have $w_{\sigma,\theta}'\in L_\rho(\R)$ if and only if $1+(\sigma-1)\rho > 0$ and $\theta(1+(\sigma-1)\rho)>1$ which is equivalent to 
\begin{align}\label{eq:proof_cond}
	\max\left\{ 1-\sigma, \frac{1-\sigma}{1-1/\theta} \right\} < \frac{1}{\rho}.
\end{align}
For $\sigma\geq 1$ this condition holds for all $\rho$. On the other hand, if $0<\sigma<1$, then $0<1-\sigma<1$ and $1/(1-1/\theta)=\theta/(\theta-1)>1$ implies that the maximum in \link{eq:proof_cond} is attained by its second entry. Hence,  $w_{\sigma,\theta}'\in L_\rho(\R)$ is equivalent to \link{cond:w'_in_L_rho}.

It remains to show assertion \link{item:w_in_B}. We split its proof into several steps.

\emph{Step 1 (Preparations).} Note that for \link{item:w_in_B} it suffices to show that $w_{\sigma,\theta}\in B^s_{\rho,\infty}(\R) \setminus B^s_{\rho,q}(\R)$ for 
$$
s=\sigma+\frac{1-1/\theta}{\rho}\qquad \text{and all} \qquad q \in\R_+.
$$
Indeed, according to \autoref{lem:embedding}\link{item:smallers}, $w_{\sigma,\theta} \in B^{s}_{\rho,\infty}(\R)$ implies $w_{\sigma,\theta} \in B^{s'}_{\rho,q}(\R)$ for all $s'<s$ and $0<q\leq \infty$. Similarly, $w_{\sigma,\theta} \in B^{s''}_{\rho,q}(\R)$ for some $s''>s$ and some $0<q\leq \infty$ would yield $w_{\sigma,\theta} \in B^{s}_{\rho,1}(\R)$.

From \link{item:w_in_L_rho} we know that $\norm{w_{\sigma,\theta} \sep L_{\rho}(\R)}<\infty$, so that it remains to prove that 
$$
\abs{w_{\sigma,\theta}}_{B^s_{\rho,q}(\R)}<\infty
\quad \text{if and only if} \quad q=\infty.
$$
To this end, note that $0<\sigma$ and $1/\theta<1$ implies $s>0$, while 
$1/\rho < \theta\,(1+\sigma)$ holds if and only if $s>1/\rho-1$.
Moreover, the assumption $1/\rho < (1-\sigma)/(1-1/\theta)$ is equivalent to $s<1$. Hence,
$$
\max\left\{0, \frac{1}{\rho}-1\right\}<s<1,
$$
so that we can use first order differences. Therefore it is enough to show that
\begin{align}\label{eq:proof_Lrhow}
	\norm{\Delta_h w_{\sigma,\theta} \sep L_\rho(\R)} 
	\sim \abs{h}^{\sigma + (1-1/\theta)/\rho}
	\qquad \text{for all} \quad h\in\R\quad \text{with}\quad \abs{h} \leq (1/6)^\theta=:\overline{t},
\end{align}
because then
$$
	t^{-s} \sup_{\substack{h\in\R,\\\abs{h}\leq t}} \norm{\Delta_h w_{\sigma,\theta} \sep L_{\rho}(\R)} \sim 1, \qquad 0<t \leq \overline{t}.
$$
Of course, we may assume w.l.o.g.\ that $h>0$. 

	\emph{Step 2 (Case $\rho=\infty$).} 
	We prove ``$\lesssim$'' for $\rho=\infty$ in \link{eq:proof_Lrhow}.
For this purpose, it suffices to show that
\begin{align}\label{eq:distw}
\abs{w_{\sigma,\theta}(x)-w_{\sigma,\theta}(y)} \leq 2 \, \abs{x-y}^\sigma \qquad \text{for all} \qquad x,y\in\R \quad \text{with} \quad x<y.
\end{align}
So let $x,y\in\R$ with $h:=y-x>0$ be fixed. Note that it is enough to consider $a_{2,\theta} \leq x<a_{\infty,\theta}$, because $x<a_{2,\theta}$ implies 
$$
\abs{w_{\sigma,\theta}(x)-w_{\sigma,\theta}(y)} = \abs{w_{\sigma,\theta}(x+h)} \leq h^\sigma \quad\text{for all} \quad h>0,
$$
while $x \geq a_{\infty,\theta}$ would lead to $w_{\sigma,\theta}(x)=w_{\sigma,\theta}(y)=0$.
For $x\in [a_{2,\theta}, a_{\infty,\theta})$ the quantity 
$$
M:=M(x,\theta):=\max\{n\geq 2 \sep a_{n,\theta}\leq x\}
$$ 
is well-defined. 
In case $h=y-x \geq M^{-\theta}$, we have
$$
\abs{w_{\sigma,\theta}(x)-w_{\sigma,\theta}(y)} 
\leq \abs{w_{\sigma,\theta}(x)}+\abs{w_{\sigma,\theta}(y)}
\leq 2\, (M^{-\theta})^\sigma \leq 2\, h^{\sigma},
$$
as claimed. So let us turn to the case $0<h<M^{-\theta}$. If $y>a_{M+1,\theta}$, then again
$w_{\sigma,\theta}(x)=0$. Moreover, in this case $a_{M+1,\theta}<y=x+h<a_{M+1,\theta}+h$, i.e.,
$$
\abs{w_{\sigma,\theta}(x)-w_{\sigma,\theta}(y)} 
= \abs{w_{\sigma,\theta}(y)}
<\, h^{\sigma}.
$$
Similarly, if $y\in [a_{M+1,\theta}-M^{-\theta},a_{M+1,\theta}]$, then $w_{\sigma,\theta}(y)=0$ and
$$
\abs{w_{\sigma,\theta}(x)-w_{\sigma,\theta}(y)} 
= \abs{w_{\sigma,\theta}(x)}
\leq\, h^{\sigma}.
$$
Hence, we are left with the case $a_{M,\theta}\leq x < y < a_{M,\theta}+3\,M^{-\theta}$ and $0<h=y-x<M^{-\theta}$, but for this situation \link{eq:distw} is obvious.

For the corresponding lower bound let $0<h<\overline{t}$. Then
\begin{align*}
	\norm{\Delta_h w_{\sigma,\theta} \sep L_\infty(\R)}
	&\geq \esssup_{x\in(a_{2,\theta}-h, a_{2,\theta})} \abs{w_{\sigma,\theta}(x+h)-w_{\sigma,\theta}(x)} \\
	&= \esssup_{y\in(a_{2,\theta}, a_{2,\theta}+h)} \abs{w_{\sigma,\theta}(y)} \\
	&=w_{\sigma,\theta}(a_{2,\theta}+h)=h^\sigma.
\end{align*}

\emph{Step 3 (Case $\rho<\infty$).} 
In order to prove \link{eq:proof_Lrhow} for $\rho<\infty$, consider the disjoint union
$$
\R=\mathscr{L}_{\theta} \cup \left( \bigcup_{n=2}^\infty \left( \mathscr{L}_{n,\theta} \cup  \mathscr{T}_{n,\theta} \cup \mathscr{R}_{n,\theta}\right)  \right) \cup \mathscr{R}_{\theta},
$$
where we set $\mathscr{L}_{\theta}:=(-\infty, 4)$, $\mathscr{R}_{\theta}:=[4\,\zeta(\theta), \infty)$, as well as
$$
\mathscr{L}_{n,\theta}:=[a_{n,\theta}, a_{n,\theta}+ \,n^{-\theta}), \quad 
\mathscr{T}_{n,\theta}:=[a_{n,\theta}+ n^{-\theta}, a_{n,\theta}+ 3\,n^{-\theta}),
$$
and $\mathscr{R}_{n,\theta}:=[a_{n,\theta}+ 3\,n^{-\theta}, a_{n,\theta}+ 4\,n^{-\theta})$ for all $n\in\N$ with $n\geq 2$. Now let $0<h \leq \overline{t}=(1/6)^\theta$ be arbitrarily fixed. Then $N(h,\theta):= \ceil{ h^{-1/\theta}/3 } \in\N$ satisfies $N(h,\theta)\geq 2$ and 
$$
N(h,\theta) - 1 \geq \frac{1}{3}\, h^{-1/\theta}-1 = h^{-1/\theta} \left[\frac{1}{3}-h^{1/\theta} \right] \geq \frac{1}{6} \, h^{-1/\theta}
$$
due to the assumption $\theta>1$. 
Further, for all $n\in\N$ with $2 \leq n\leq N(h,\theta)$ it holds $n \leq h^{-1/\theta}/3+1$, i.e., 
$$
0<h\leq \frac{1}{3^\theta(n-1)^{\theta}} \leq (n+1)^{-\theta}.
$$
In this case
\begin{align}
	\norm{\Delta_h w_{\sigma,\theta} \sep L_\rho(\mathscr{R}_{n,\theta})}^\rho
	&= \int_{a_{n,\theta}+3\, n^{-\theta}}^{a_{n,\theta}+4\, n^{-\theta}} \abs{w_{\sigma,\theta}(\xi+h)-w_{\sigma,\theta}(\xi)}^\rho \d \xi \nonumber\\
	&= \int_{a_{n+1,\theta}}^{a_{n+1,\theta}+h} w_{\sigma,\theta}(y)^\rho \d y \nonumber\\
	&=\int_{0}^{h} y^{\sigma\rho} \d y \nonumber\\
	&= \frac{1}{\sigma\rho+1} \, h^{\sigma\rho+1} \label{eq:proof_NormR}
\end{align}
which yields the desired lower bound
\begin{align*}
	\norm{\Delta_h w_{\sigma,\theta} \sep L_\rho(\R)}^\rho
	&\geq \norm{\Delta_h w_{\sigma,\theta} \sep L_\rho\left(\bigcup_{n=2}^{N(h,\theta)} \mathscr{R}_{n,\theta} \right)}^\rho \\
	&= \sum_{n=2}^{N(h,\theta)} \norm{\Delta_h w_{\sigma,\theta} \sep L_\rho\left( \mathscr{R}_{n,\theta} \right)}^\rho \\
	&\geq (N(h,\theta)-1)\,\frac{1}{\sigma\rho+1} \, h^{\sigma\rho+1} \\
	&\gtrsim h^{\sigma\rho+1-1/\theta}.
\end{align*}

Let us show the corresponding upper bound. Using Step 2 and log-convexity of $L_\rho$-norms, we see that the bound for $L_\rho$ implies the respective bound for all $L_p$ with $0<1/p<1/\rho$:
\begin{align*}
	\norm{\Delta_h w_{\sigma,\theta} \sep L_p(\R)}^p
	&\leq \left( \norm{\Delta_h w_{\sigma,\theta} \sep L_\rho(\R)}^{\rho/p} \, \norm{\Delta_h w_{\sigma,\theta} \sep L_\infty(\R)}^{1-\rho/p} \right)^p \\
	&= \norm{\Delta_h w_{\sigma,\theta} \sep L_\rho(\R)}^\rho \, \norm{\Delta_h w_{\sigma,\theta} \sep L_\infty(\R)}^{p-\rho} \\
	&\lesssim h^{\rho\sigma + 1-1/\theta} \, \left( h^{\sigma} \right)^{p-\rho}\\
	&=h^{p\sigma + 1-1/\theta}.
\end{align*}
Therefore, since $(1-\sigma)/(1-1/\theta)>1$ if and only if $\sigma < 1/\theta$, we may assume w.l.o.g. 
$$
	1 \leq \frac{1}{\rho} < \frac{1-\sigma}{1-1/\theta}.
$$
So, let $0<\rho \leq 1$ and consider $2\leq n \leq N(h,\theta)$. Then H\"older's inequality (with $1/r:=1-\rho$ and $1/r'=\rho$) and the monotonicity of $w_{\sigma,\theta}$ imply
\begin{align*}
	\norm{\Delta_h w_{\sigma,\theta} \sep L_\rho(\mathscr{L}_{n,\theta})}
	&\leq \left( n^{-\theta} \right)^{1/\rho-1}\, \int_{a_{n,\theta}}^{a_{n,\theta}+n^{-\theta}} \big( w_{\sigma,\theta}(\xi+h)-w_{\sigma,\theta}(\xi) \big) \d \xi \\
	&=  n^{-\theta(1/\rho-1)} \, \left( \int_{a_{n,\theta}+h}^{a_{n,\theta}+n^{-\theta}+h} w_{\sigma,\theta}(y) \d y - \int_{a_{n,\theta}}^{a_{n,\theta}+n^{-\theta}} w_{\sigma,\theta}(\xi) \d \xi \right) \\
	&= n^{-\theta(1/\rho-1)} \, \left( h\, n^{-\theta\sigma} - \int_{a_{n,\theta}}^{a_{n,\theta}+h} w_{\sigma,\theta}(\xi) \d \xi \right)\\
	&\leq h\, n^{-\theta(\sigma+1/\rho-1)},
\end{align*}
i.e., $\norm{\Delta_h w_{\sigma,\theta} \sep L_\rho(\mathscr{L}_{n,\theta})}^\rho \leq h^\rho \, n^{-\theta(\sigma\rho+1-\rho)}$, as well as
\begin{align*}
	\norm{\Delta_h w_{\sigma,\theta} \sep L_\rho(\mathscr{T}_{n,\theta})}
	&\leq \left( 2n^{-\theta} \right)^{1/\rho-1}\, \int_{a_{n,\theta}+ n^{-\theta}}^{a_{n,\theta}+3\,n^{-\theta}} \big( w_{\sigma,\theta}(\xi+h)-w_{\sigma,\theta}(\xi) \big) \d \xi \\ 
	&\lesssim n^{-\theta(1/\rho-1)} \left( \int_{a_{n,\theta}+ n^{-\theta}}^{a_{n,\theta}+ 3\,n^{-\theta}} w_{\sigma,\theta}(\xi) \d \xi - \int_{a_{n,\theta}+ n^{-\theta}+h}^{a_{n,\theta}+ 3\,n^{-\theta}+h} w_{\sigma,\theta}(\xi) \d \xi \right) \\ 
	&= n^{-\theta(1/\rho-1)} \left( \int_{a_{n,\theta}+ n^{-\theta}}^{a_{n,\theta}+ n^{-\theta}+h} \underbrace{ w_{\sigma,\theta}(\xi)}_{=n^{-\theta\sigma}} \d \xi - \int_{a_{n,\theta}+ 3\,n^{-\theta}}^{a_{n,\theta}+ 3\,n^{-\theta}+h} \underbrace{w_{\sigma,\theta}(\xi)}_{=0} \d \xi \right)  \\
	&= h\, n^{-\theta(\sigma+1/\rho-1)}
\end{align*}
so that $\norm{\Delta_h w_{\sigma,\theta} \sep L_\rho(\mathscr{T}_{n,\theta})}^\rho \lesssim h^\rho \, n^{-\theta(\sigma\rho+1-\rho)}$. 
Further we have $\sigma\rho+1-\rho >0$ because $1/\rho \geq 1 \geq 1-\sigma$. So, \link{eq:proof_NormR} and $h\leq n^{-\theta}$ yield that also
\begin{align*}
	\norm{\Delta_h w_{\sigma,\theta} \sep L_\rho(\mathscr{R}_{n,\theta})}^\rho \sim h^{\sigma\rho+1} 
	= h^{\rho}\, h^{\sigma\rho+1-\rho} 
	\lesssim h^\rho \, n^{-\theta(\sigma\rho+1-\rho)}.
\end{align*}
Combining the latter estimates shows that
\begin{align}\label{eq:proof_boundsum}
	\norm{\Delta_h w_{\sigma,\theta} \sep L_\rho \left( \bigcup_{n=2}^{N(h,\theta)} \left( \mathscr{L}_{n,\theta} \cup  \mathscr{T}_{n,\theta} \cup \mathscr{R}_{n,\theta}\right)  \right) }^\rho
	&\lesssim h^\rho\, \sum_{n=2}^{N(h,\theta)} n^{-\theta(\sigma\rho+1-\rho)}
	\quad \text{for all} \quad 0<\rho\leq 1. 
\end{align}
Now additionally assume 
$1/\rho < (1-\sigma)/(1-1/\theta)$. Then there holds $0<\theta(\sigma\rho+1-\rho)<1$ and hence
\begin{align*}
	\sum_{n=2}^{N(h,\theta)} n^{-\theta(\sigma\rho+1-\rho)} 
	&\leq \int_1^{N(h,\theta)} x^{-\theta(\sigma\rho+1-\rho)}\d x 
	=\frac{1}{1-\theta(\sigma\rho+1-\rho)} \left( N(h,\theta)^{1-\theta(\sigma\rho+1-\rho)} -1\right),
\end{align*}
where
\begin{align*}
	N(h,\theta)^{1-\theta(\sigma\rho+1-\rho)} -1 
	&\leq \left( \frac{1}{3}\,h^{-1/\theta}+1 \right)^{1-\theta(\sigma\rho+1-\rho)} -1 \\
	&\leq \left( \frac{1}{3}\,h^{-1/\theta} \right)^{1-\theta(\sigma\rho+1-\rho)} 
	\\
	&
	\sim h^{-\rho+\sigma\rho+1-1/\theta}.
\end{align*}
Therefore, we arrive at
\begin{align*}
	\norm{\Delta_h w_{\sigma,\theta} \sep L_\rho \left( \bigcup_{n=2}^{N(h,\theta)} \left( \mathscr{L}_{n,\theta} \cup  \mathscr{T}_{n,\theta} \cup \mathscr{R}_{n,\theta}\right)  \right) }^\rho
	&\lesssim h^{\sigma\rho+1-1/\theta}. 
\end{align*}
Moreover, \link{eq:distw} and $\theta>1$ yield that also
\begin{align*}
	&\norm{\Delta_h w_{\sigma,\theta} \sep L_\rho \left( \bigcup_{n=N(h,\theta)+1}^\infty \left( \mathscr{L}_{n,\theta} \cup  \mathscr{T}_{n,\theta} \cup \mathscr{R}_{n,\theta}\right)  \right) }^\rho 
	\\
	&\qquad
	=\sum_{n=N(h,\theta)+1}^\infty \norm{\Delta_h w_{\sigma,\theta} \sep L_\rho([a_{n,\theta},a_{n,\theta}+4\,n^{-\theta}])}^\rho 
\end{align*}
is bounded by
\begin{align*}	
	\sum_{n=N(h,\theta)+1}^\infty 4\,n^{-\theta} \, (2 \,h^\sigma)^\rho 
	\lesssim h^{\sigma\rho}\, \int_{N(h,\theta)}^\infty x^{-\theta} \d x
	= h^{\sigma\rho}\, \frac{1}{\theta-1} \,  N(h,\theta)^{1-\theta} 
	\lesssim h^{\sigma\rho+1-1/\theta}. 
\end{align*}
Finally, we clearly have $w_{\sigma,\theta}(x)=0$ on $\mathscr{L}_{\theta}\cup \mathscr{R}_{\theta}$ and hence $\norm{\Delta_h w_{\sigma,\theta} \sep L_\rho(\mathscr{R}_{\theta})}=0$, as well as
\begin{align*}
	\norm{\Delta_h w_{\sigma,\theta} \sep L_\rho(\mathscr{L}_{\theta})}^\rho 
	&= \int_{a_{2,\theta}}^{a_{2,\theta}+h} w_{\sigma,\theta}(y)^\rho \d y 
	\leq \int_{0}^{h} x^{\sigma\rho} \d x
	= \frac{1}{\sigma\rho+1} \, h^{\sigma\rho+1} \lesssim
h^{\sigma\rho+1-1/\theta}.
\end{align*}
Altogether, this shows \link{eq:proof_Lrhow} and thus the proof is complete.
\end{proof}

\begin{remark}\label{rem:parameter}
We stress that some parameter restrictions in \autoref{lem:Regw}\link{item:w_in_B} are stronger than required. If $\rho=\infty$, our proof actually works for all $0<\sigma<1<\theta<\infty$. Moreover, the upper bound on $1/\rho$ in \link{cond:rho} seems to be an artifact of our proof technique. At least for the ``only if'' part it can be dropped, as can be seen easily using complex interpolation. 
\hfill$\square$
\end{remark}

In order to proceed, again let $\sigma\in\R_+$ and $1<\theta<\infty$. Then, based on $w_{\sigma,\theta}$ as defined above, let us set
\begin{align}\label{eq:def_u}
	\begin{aligned}
		v_{\sigma,\theta}\colon \R_+\to\R,
		\qquad 
		&r\mapsto v_{\sigma,\theta}(r) := w_{\sigma,\theta} \big( 16\,\zeta(\theta)\,r-4\,\zeta(\theta) \big) - w_{\sigma,\theta}\big(16\,\zeta(\theta)\,r-8\,\zeta(\theta)\big), \\
		u_{\sigma,\theta}\colon \R_+\to\R, 
		\qquad 
		&r\mapsto  u_{\sigma,\theta}(r):=\int_{0}^{r}  v_{\sigma,\theta}(\xi)\d \xi. 
	\end{aligned}
\end{align}

\begin{lemma}[Properties of $v_{\sigma,\theta}$ and $u_{\sigma,\theta}$]\label{lem:prop_uv}
	Let $\sigma\in \R_+$ and $1<\theta<\infty$. Then
\begin{enumerate}[label=(\roman*),ref=\roman*]
	\item\label{item:supp_uv} the supports of $ u_{\sigma,\theta}$ and $ v_{\sigma,\theta}$ are contained in $\tilde{\mathscr{S}}_\theta:= [ 1/4, 3/4]$.
	\item \label{item:vgammasigma} $ v_{\gamma\sigma,\theta}=\abs{ v_{\sigma,\theta}}^{\gamma-1}  v_{\sigma,\theta}$ for all $\gamma \in\R_+$.
	\item\label{item:u'} $u_{\sigma,\theta}\in C^1(\R_+)$ with $u'_{\sigma,\theta}=v_{\sigma,\theta}$.
	\item\label{item:Regv} for all $0<\rho,q\leq \infty$ and $s<1$ we have
	\begin{align*}
		u_{\sigma,\theta}\in B_{\rho,q}^{1+s}(\R_+) \qquad\text{if and only if}\qquad w_{\sigma,\theta}\in B_{\rho,q}^s(\R).
	\end{align*} 	
	\item\label{item:Regv2} for all $1<\rho<\infty$ we have
		\begin{align*}
			u_{\sigma,\theta}\in W^{2}_{\rho}(\R_+) \qquad\text{if and only if}\qquad w_{\sigma,\theta}\in W^1_{\rho}(\R).
		\end{align*} 	
\end{enumerate}
\end{lemma}
\begin{proof}
We use $\supp(w_{\sigma,\theta}) \subseteq \mathscr{S}_\theta=[a_{2,\theta},a_{\infty,\theta}]=[4,4\,\zeta(\theta)]$, as shown in \autoref{lem:prop_w}\link{item:supp_w}, to deduce the representation
\begin{align}\label{eq:representation_v}
	v_{\sigma,\theta}(r) 
	= \begin{cases}
		w_{\sigma,\theta}(t) &\quad\text{if}\quad\displaystyle r=\frac{t+4\,\zeta(\theta)}{16\,\zeta(\theta)} \in \left[\frac{1}{4}+\frac{1}{4\,\zeta(\theta)}, \frac{1}{2} \right],\\
		-w_{\sigma,\theta}(t') &\quad\text{if}\quad\displaystyle r=\frac{t'+8\,\zeta(\theta)}{16\,\zeta(\theta)} \in \left[\frac{1}{2}+\frac{1}{4\,\zeta(\theta)}, \frac{3}{4} \right],\\
		0&\quad\text{else}.
	\end{cases}
\end{align}
This proves \link{item:supp_uv} for $ v_{\sigma,\theta}$.
Moreover, for $0< r < 1/4$ we have
$
 u_{\sigma,\theta}(r)=\int_0^r 0 \d \xi =0,
$
while for $r>3/4$ we may write
\begin{align*}
	u_{\sigma,\theta}(r)
	&=\int_0^r  v_{\sigma,\theta}(\xi)\d \xi 
	=0+\int_{1/4}^{1/2} v_{\sigma,\theta}(\xi)\d \xi 
		-\int_{1/2}^{3/4} -v_{\sigma,\theta}(\xi) \d \xi
		+0
		=0 
\end{align*}
which shows \link{item:supp_uv} for $u_{\sigma,\theta}$. 
Further, \link{item:vgammasigma} directly follows from \link{eq:representation_v} and \autoref{lem:prop_w}\link{item:wgammasigma}.

We are left with proving the regularity assertions \link{item:u'}--\link{item:Regv2}. The fact that $u_{\sigma,\theta}\in C^1(\R_+)$ with $u_{\sigma,\theta}'=v_{\sigma,\theta}$ is a consequence of the fundamental theorem of calculus and the continuity of $v_{\sigma,\theta}$, cf.\ \link{eq:def_u} and \autoref{lem:prop_w}\link{item:supp_w} again. This shows \link{item:u'}.

If we assume that $w_{\sigma,\theta}\in B_{\rho,q}^s(\R)$, then also
$$
	\widetilde{v_{\sigma,\theta}} 
	:= w_{\sigma,\theta}(16\,\zeta(\theta)\,\cdot-4\,\zeta(\theta)) - w_{\sigma,\theta}(16\,\zeta(\theta)\,\cdot-8\,\zeta(\theta)) \in B_{\rho,q}^s(\R)
$$
because $B_{\rho,q}^s(\R)$ is invariant under diffeomorphic coordinate transformations; see, e.g., Triebel \cite[Sect.\ 2.10.2]{T83}. 
Since $v_{\sigma,\theta} = \widetilde{v_{\sigma,\theta}} \big|_{\R_+}$ this yields $v_{\sigma,\theta} \in B_{\rho,q}^s(\R_+)$. On the other hand, $v_{\sigma,\theta} \in B_{\rho,q}^s(\R_+)$ implies that there exists $g\in B_{\rho,q}^s(\R)$ such that $v_{\sigma,\theta} = g \big|_{\R_+}$. Now let $\chi\in C^\infty(\R)$ with
$$
	\chi(x) 
	= \begin{cases}
		1 & \text{for}\quad  1/4 \leq x \leq 1/2+1/(6\,\zeta(\theta)), \\
		0 & \text{for}\quad x \leq 1/5, \quad\text{or}\quad 1/2+1/(5\,\zeta(\theta))\leq x.
	\end{cases}
$$
Then, according to a multiplication theorem by Triebel \cite[Sect.\ 4.2.2]{T92}, we conclude that $w_{\sigma,\theta}=\chi\, g \in B_{\rho,q}^s(\R)$. 
Due to \link{item:u'}, this shows that for all $0<\rho,q\leq \infty$ and $s\in\R$
	\begin{align}\label{eq:Regv}
u_{\sigma,\theta}'\in B_{\rho,q}^s(\R_+) \qquad\text{if and only if}\qquad w_{\sigma,\theta}\in B_{\rho,q}^s(\R).
\end{align} 
Moreover, we may extend $u_{\sigma,\theta}\in C^1(\R_+)$ by zero in order to obtain $\tilde{u_{\sigma,\theta}}\in C^1(\R)$.
Using the characterization of Besov spaces in terms of first order differences, we see that this gives $\tilde{u_{\sigma,\theta}} \in B^{1-\epsilon}_{\infty,q}(\R)$ for all $0<\epsilon<1$. Choosing $\epsilon$ small enough such that $s<1-\epsilon$ then shows $\tilde{u_{\sigma,\theta}} \in B^{s}_{\infty,q}(\R) \hookrightarrow B^{s}_{\rho,q}(\R)$, i.e., $u_{\sigma,\theta} \in B^{s}_{\rho,q}(\R_+)$, where we used \autoref{lem:embedding} and the compact support of $\tilde{u_{\sigma,\theta}}$. Therefore, \link{item:Regv} follows from \autoref{lem:embedding}\link{item:lift}.

Since Sobolev spaces $W^k_\rho$ can be identified as special Triebel-Lizorkin spaces $F^k_{\rho,2}$, we can argue similarly for this case. Instead of \link{eq:Regv} we now have that for every $k\in\N_0$ (particularly for $k=1$) and $1<\rho<\infty$ there holds $u_{\sigma,\theta}'\in W^k_{\rho}(\R_+)$ if and only if $w_{\sigma,\theta}\in W^k_{\rho}(\R)$. 
Further from \link{item:supp_uv} and \link{item:u'} we clearly have $u_{\sigma,\theta}\in W^1_{\rho}(\R_+)$. Together this shows \link{item:Regv2} and hence the proof is complete.
\end{proof}

Next, for $d\in \N$ let $\Omega \subseteq \R^d$ be either $\R^d$ itself, a bounded Lipschitz domain, or an interval (if $d=1$), and assume that $\Omega$ contains the Euclidean unit ball $B_1(0):=\{x\in\R^d \sep \abs{x}_2<1\}$. Given $\sigma \in \R_+$, as well as $1<p,\theta<\infty$, we then let 
\begin{align}\label{eq:def_uf}
	u_{\sigma,\theta,d}(\psi):=\int_\Omega u_{\sigma,\theta}(\abs{x}_2) \, \psi(x)\d x
	\quad\text{and}\quad
	 f_{\sigma,\theta,d}^{[p]}(\psi) := \int_\Omega v_{(p-1)\sigma,\theta}(\abs{x}_2) \, \distr{\frac{x}{\abs{x}_2}}{\nabla \psi(x)}_{\R^d}\d x
\end{align}
for all test functions $\psi\in\D(\Omega)$. It is easily seen that we actually deal with distributions $u_{\sigma,\theta,d}, f_{\sigma,\theta,d}^{[p]}\in \D'(\Omega)$ since \autoref{lem:prop_uv}\link{item:u'} implies $u_{\sigma,\theta}, v_{(p-1)\sigma,\theta}\in L_\infty(\R_+)$. These distributions are closely related:

\begin{lemma}\label{lem:p-poisson}
Let $d\in\N$ and $1<p,\theta<\infty$, as well as $\sigma\in\R_+$. Further let $\Omega \subseteq \R^d$ be either $\R^d$ itself, a bounded Lipschitz domain, or an interval (if $d=1$), and assume $B_1(0)\subseteq\Omega$. Then 
	\begin{enumerate}[label=(\roman*),ref=\roman*]
		\item\label{item:supp_U} $\supp(u_{\sigma,\theta,d}), \supp(f^{[p]}_{\sigma,\theta,d})\subseteq B_{4/5}(0)$,
		\item\label{item:Div} for all $x\in\Omega$ there holds
		\begin{align}\label{eq:Anablau}
			 A(\nabla u_{\sigma,\theta,d})(x) 
			 = v_{(p-1)\sigma,\theta}(\abs{x}_2) \, \frac{x}{\abs{x}_2}
			 = \nabla u_{(p-1)\sigma,\theta,d}(x),
		\end{align}
	\item\label{item:ppoisson} $f_{\sigma,\theta,d}^{[p]} \in W_{p'}^{-1}(\Omega)$ and $u_{\sigma,\theta,d}\in W^1_{p,0}(\Omega)$ constitutes a weak solution $u$ to
	\begin{align*}
	-\Delta_p u = f_{\sigma,\theta,d}^{[p]} \qquad \text{and} \qquad 
u|_{\partial\Omega}=0.
\end{align*}
	\end{enumerate}
\end{lemma}
\begin{proof}
	In view of \autoref{lem:prop_uv} assertion \link{item:supp_U} is obvious. Moreover, it is clear that $u_{\sigma,\theta,d}\in \D'(\Omega)$ is regular and can be identified with the function $\big( u_{\sigma,\theta}\circ r_d \big)\big|_\Omega\in C^1(\Omega)$, where we set
\begin{align}\label{eq:def_rd}
	r_d \colon \R^d\to\R, \qquad x\mapsto r_d(x):=\abs{x}_2.
\end{align} 
With this interpretation we have 
	$$
	\frac{\partial u_{\sigma,\theta,d}}{\partial x_j}(x) = 0 = v_{\sigma,\theta}(\abs{x}_2) \, \frac{x_j}{\abs{x}_2}, \qquad j=1,\ldots,d,
	$$
	for all $x\in B_{1/4}(0)$, while on $\Omega\setminus B_{1/8}(0)$ the chain rule and \autoref{lem:prop_uv}\link{item:u'} give
	$$
		\frac{\partial u_{\sigma,\theta,d}}{\partial x_j}
		= \frac{\partial (u_{\sigma,\theta}\circ r_d)}{\partial x_j}
		= v_{\sigma,\theta}(r_d(\cdot))\,\frac{\partial r_d}{\partial x_j},
	$$
	where 
	$$
		\frac{\partial r_d}{\partial x_j}(x)
		= \frac{\partial}{\partial x_j} \left[ \left( \sum_{k=1}^d x_k^2 \right)^{1/2}\right](x)
		=\frac{1}{2} \left( \sum_{k=1}^d x_k^2 \right)^{-1/2} \, 2x_j = \frac{x_j}{\abs{x}_2}.
	$$
	Together this shows 
	\begin{align}\label{eq:proof_gradu}
		\frac{\partial u_{\sigma,\theta,d}}{\partial x_j}(x)= v_{\sigma,\theta}(\abs{x}_2) \, \frac{x_j}{\abs{x}_2}
		\qquad \text{for all} \qquad j=1,\ldots,d \quad \text{and} \quad x\in\Omega
	\end{align}	
and hence
\begin{align*}
	\norm{u_{\sigma,\theta,d} \sep W^1_p(\Omega)} 
	= \norm{u_{\sigma,\theta,d} \sep L_p(\Omega)} + \sum_{j=1}^d	\norm{\frac{\partial u_{\sigma,\theta,d}}{\partial x_j} \sep L_p(\Omega)} <\infty
\end{align*}
since $u_{\sigma,\theta}, v_{\sigma,\theta}\in L_\infty(\R_+)$ with compact support and $\abs{x_j / \abs{x}_2}\leq 1$. So, we can conclude $u_{\sigma,\theta,d} \in W^1_{p,0}(\Omega)$.
Further, as a direct consequence of \link{eq:proof_gradu}, we obtain
	\begin{align*}
		\abs{\nabla u_{\sigma,\theta,d}(x)}_2 
		= \left(\sum_{j=1}^d \abs{\frac{\partial u_{\sigma,\theta,d}}{\partial x_j}(x)}^2 \right)^{1/2}
		\stackrel{\text{\link{eq:proof_gradu}}}{=} \left( \frac{\abs{v_{\sigma,\theta}(\abs{x}_2)}^2}{\abs{x}_2^2}  \sum_{j=1}^d \abs{x_j}^2 \right)^{1/2}
		= \abs{v_{\sigma,\theta}(\abs{x}_2)}
	\end{align*}
	such that by \autoref{lem:prop_uv}\link{item:vgammasigma} with $\gamma:=p-1$ we have that
	\begin{align*}
		A(\nabla u_{\sigma,\theta,d})(x)
		&=\abs{\nabla u_{\sigma,\theta,d}(x)}_2^{p-2}\,\nabla u_{\sigma,\theta,d}(x) \\ 
		&= \abs{v_{\sigma,\theta}(\abs{x}_2)}^{(p-1)-1} \, v_{\sigma,\theta}(\abs{x}_2) \, \frac{x}{\abs{x}_2} \\
		&= v_{(p-1)\sigma,\theta}(\abs{x}_2) \, \frac{x}{\abs{x}_2}
	\end{align*}
	holds for all $x\in\Omega$. Together with \link{eq:proof_gradu} this proves \link{eq:Anablau}, as well as
	\begin{align}\label{eq:representation_f}
		\int_{\Omega} \distr{A(\nabla u_{\sigma,\theta,d})(x)}{\nabla \psi(x)}_{\R^d} \d x
		= \int_\Omega v_{(p-1)\sigma,\theta}(\abs{x}_2) \, \distr{\frac{x}{\abs{x}_2}}{\nabla \psi(x)}_{\R^d}\d x
		= f_{\sigma,\theta,d}^{[p]}(\psi)
	\end{align}
	for each $\psi\in\D(\Omega)$. In other words, there holds $-\Delta_p(u_{\sigma,\theta,d})= f_{\sigma,\theta,d}^{[p]}$ in the weak sense.
	Finally, H\"older's inequality on $B_1(0)\subseteq \Omega$ proves
	\begin{align*}
		\abs{f_{\sigma,\theta,d}^{[p]}(\psi)}
		&\leq \int_\Omega \abs{v_{(p-1)\sigma,\theta}(\abs{x}_2)}\, \abs{\nabla \psi(x)}_2 \d x \\
		&\leq \norm{v_{(p-1)\sigma,\theta} \sep L_\infty(\R_+)} \, \int_{B_1(0)} \abs{\nabla \psi(x)}_2 \d x \\
		&\lesssim \norm{v_{(p-1)\sigma,\theta} \sep L_\infty(\R_+)} \, \norm{\psi \sep W_p^1(\Omega)}, \qquad \psi\in\D(\Omega).
	\end{align*}
	Since by definition $\D(\Omega)$ is dense $W^1_{p,0}(\Omega)$ we therefore have $f_{\sigma,\theta,d}^{[p]} \in (W^1_{p,0}(\Omega))'=W^{-1}_{p'}(\Omega)$ and the proof is complete.
\end{proof}

In order to provide further regularity assertions for $u_{\sigma,\theta,d}$ and $f_{\sigma,\theta,d}^{[p]}$, we will need the subsequent result which characterizes the smoothness and integrability of rotationally invariant functions. Therein $r_d$ has the same meaning as in \link{eq:def_rd}.

\begin{prop}\label{lem:rotation_reg}
	Let $0<a<b<\infty$. Assume that $g\colon \R_+ \to \C$ is measurable with support $\supp(g)\subseteq [a,b]$ and let $g_d:=g \circ r_d$. Then
\begin{enumerate}[label=(\roman*),ref=\roman*]
\item\label{item:Gwelldefined} $g_d \colon\R^d \to\C$ is well-defined almost everywhere.
\item\label{item:RotRegular} for $0<\rho,q\leq\infty$ and $s > d\, \max\{0, 1/\rho-1\}$ there holds
\begin{align}\label{equiv:Lp}
	g_d \in L_{\rho}(\R^d) 
	\qquad \text{if and only if} \qquad 
	g \in L_{\rho}(\R_+),
\end{align}
as well as
\begin{align*}
	g_d \in B^s_{\rho,q}(\R^d) 
	\qquad \text{if and only if} \qquad 
	g \in B^s_{\rho,q}(\R_+).
\end{align*}
\item\label{item:RotRegularSobolev} for $1<\rho<\infty$ and $k\in\N$ there holds
\begin{align*}
	g_d \in W^k_{\rho}(\R^d) 
	\qquad \text{if and only if} \qquad 
	g \in W^k_{\rho}(\R_+).
\end{align*}
\end{enumerate}
\end{prop}
\begin{proof}
With $g$ and $r_d$ also $g_d$ is measurable such that it can be represented as an a.e.\ convergent pointwise limit of simple functions.
This proves \link{item:Gwelldefined}.
	
If $d=1$, then the equivalences in \link{item:RotRegular} and \link{item:RotRegularSobolev} are trivial, as $g$ vanishes in a neighborhood of the origin. So let us assume $d\geq 2$. Then for $\rho<\infty$ the first assertion in \link{item:RotRegular} follows from a simple transformation into (generalized) polar coordinates $x=r\,\vartheta(\phi)$, $(r,\phi) \in [0,\infty)\times \Phi$:
\begin{align*}
	\norm{g_d \sep L_{\rho}(\R^d)}^\rho 
	= \int_{\R^d} \abs{g(r_d(x))}^\rho \d x 
	&= \int_\Phi \int_a^b \abs{g(r)}^\rho \, r^{d-1} T(\phi) \, \d r \, \d \phi 
	&\sim 
	\norm{g \sep L_{\rho}((0,\infty))}^\rho,  
\end{align*} 
where we used that $\supp(g)\subseteq [a,b]$ and that $T$ is some tensor product of trigonometric functions defined on $\Phi \subset [-\pi,\pi]^{d-1}$. For $\rho=\infty$ the equivalence \link{equiv:Lp} is obvious.
	
It remains to prove the equivalences for multivariate Besov and Sobolev spaces. 
In case of $B^s_{\rho,q}$ and $\rho=\infty$, this follows from results due to Sickel \emph{et al.} \cite[Thm.\ 2]{SicSkrVyb2012}, while the case $0<\rho<\infty$ is covered by \cite[Cor.~1~\&~2]{SicSkrVyb2012}. 
However, \cite[Cor.~1~\&~2]{SicSkrVyb2012} also covers the assertion for Sobolev spaces since $W^k_\rho=F^k_{\rho,2}$ if $1<\rho<\infty$.
\end{proof}

\begin{lemma}[Regularity of $u_{\sigma,\theta,d}$, $A(\nabla u_{\sigma,\theta,d})$, and $f^{[p]}_{\sigma,\theta,d}$]\label{lem:reg_u}
Let $d\in\N$ and $1<p,\theta<\infty$, as well as $\sigma\in\R_+$. Further let $\Omega \subseteq \R^d$ be either $\R^d$ itself, a bounded Lipschitz domain, or an interval (if $d=1$), and assume $B_1(0)\subseteq\Omega$. Moreover, let $0<\rho,q\leq\infty$ and $s\in\R$ with $d\, \max\{ 0, 1/\rho-1 \} < s < 1$.
Then
	\begin{enumerate}[label=(\roman*),ref=\roman*]
		\item \label{item:reg_ud} there holds 
			$$
				u_{\sigma,\theta,d} \in B^{1+s}_{\rho,q}(\Omega) 
				\qquad\text{if and only if}\qquad
				w_{\sigma,\theta} \in B^{s}_{\rho,q}(\R)
			$$
			and $1<\rho<\infty$ implies that
			$$
				u_{\sigma,\theta,d} \in W^2_{\rho}(\Omega) 
				\qquad\text{if and only if}\qquad
				w_{\sigma,\theta} \in W^1_{\rho}(\R).
			$$
		\item \label{item:reg_A} there holds 
		$$
			A(\nabla u_{\sigma,\theta,d}) \in \big( B^{s}_{\rho,q}(\Omega) \big)^d
			\qquad\text{if and only if}\qquad
			w_{(p-1)\sigma,\theta} \in B^{s}_{\rho,q}(\R)
		$$
		and $1<\rho<\infty$ implies that
		$$
			A(\nabla u_{\sigma,\theta,d}) \in \big( W^1_{\rho}(\Omega) \big)^d
			\qquad\text{if and only if}\qquad
			w_{(p-1)\sigma,\theta} \in W^1_{\rho}(\R).
		$$
	\end{enumerate}
	Additionally  assume that $\min\{\rho,q\}>1$. Then
	\begin{enumerate}[label=(\roman*),ref=\roman*, resume]	
		\item \label{item:f_in_B} $A(\nabla u_{\sigma,\theta,d}) \in \big( B^{s}_{\rho,q}(\Omega) \big)^d$ implies $f_{\sigma,\theta,d}^{[p]} \in B^{-1+s}_{\rho,q}(\Omega)$,
		\item \label{item:f_in_Lrho} $w'_{(p-1)\sigma,\theta} \in L_{\rho}(\R)$ implies $f_{\sigma,\theta,d}^{[p]} \in L_{\rho}(\Omega)$.
	\end{enumerate}
\end{lemma}
\begin{proof}
	Recall that $u_{\sigma,\theta,d}\in\D'(\Omega)$ can be identified with the function $u_{\sigma,\theta}\circ r_d$ restricted to $\Omega$. Hence, $u_{\sigma,\theta}\circ r_d\in B^{1+s}_{\rho,q}(\R^d)$ implies $u_{\sigma,\theta,d}\in B^{1+s}_{\rho,q}(\Omega)$. On the other hand, if $u_{\sigma,\theta,d}\in B^{1+s}_{\rho,q}(\Omega)$, then there exists $\widetilde{u}\in B^{1+s}_{\rho,q}(\R^d)$ with $\widetilde{u}\big|_{\Omega}=u_{\sigma,\theta,d}$. Now let $\chi\in C^{\infty}(\R^d)$ with
	$$
		\chi(x)= \begin{cases}
			1 &\quad\text{if}\quad x\in B_{4/5}(0),\\
			0 &\quad\text{if}\quad x\in \R^d\setminus B_1(0).
		\end{cases}
	$$ 
	Then $u_{\sigma,\theta}\circ r_d=\chi\,\widetilde{u} \in B^{1+s}_{\rho,q}(\R^d)$ due to \cite[Sect.\ 4.2.2]{T92}. Therefore, $u_{\sigma,\theta,d}\in B^{1+s}_{\rho,q}(\Omega)$ is equivalent to $u_{\sigma,\theta}\circ r_d\in B^{1+s}_{\rho,q}(\R^d)$. By \autoref{lem:prop_uv}\link{item:supp_uv} and \autoref{lem:rotation_reg}\link{item:RotRegular} this holds if and only if $u_{\sigma,\theta} \in B^{1+s}_{\rho,q}(\R_+)$. 	
	Since we assume $s<1$, we can use  \autoref{lem:prop_uv}\link{item:Regv} to see that this is equivalent to $w_{\sigma,\theta} \in B^{s}_{\rho,q}(\R)$. Thus, we have shown \link{item:reg_ud} in the case of Besov spaces. 
	For Sobolev spaces we can argue similarly.
	
	Next we apply \link{item:reg_ud} to deduce that $w_{(p-1)\sigma,\theta} \in B^{s}_{\rho,q}(\R)$ is equivalent to $u_{(p-1)\sigma,\theta,d} \in B^{1+s}_{\rho,q}(\Omega)$. By \autoref{lem:embedding}\link{item:lift} this holds if and only if $u_{(p-1)\sigma,\theta,d} \in B^{s}_{\rho,q}(\Omega)$ and $\nabla u_{(p-1)\sigma,\theta,d} \in (B^{s}_{\rho,q}(\Omega))^d$. 
	Since we assume that $s<1$, the first condition is always fulfilled (cf.\ the proof of \autoref{lem:prop_uv}!), and by \autoref{lem:p-poisson}\link{item:Div} $\nabla u_{(p-1)\sigma,\theta,d}$ is nothing but
	$A(\nabla u_{\sigma,\theta,d})$.
	Also here the proof for Sobolev spaces is essentially the same.

	Let us prove \link{item:f_in_B}. To this end, we note that $A(\nabla u_{\sigma,\theta,d}) \in \big( B^{s}_{\rho,q}(\Omega) \big)^d$ implies that for every $j=1,\ldots,d$ we have
	\begin{align*}
		A(\nabla u_{\sigma,\theta,d})_j \in \widetilde{B}^{s}_{\rho,q}(\Omega):=\big\{ g \in \D'(\Omega) \sep \exists\, \widetilde{g}\in B^{s}_{\rho,q}(\R^d) \text{ with } g=\widetilde{g}\big|_{\Omega} \text{ and } \supp(\widetilde{g})\subseteq \overline{\Omega} \big\}
	\end{align*}
	according to \autoref{lem:p-poisson}, where we set 
	\begin{align*}
		\norm{g \sep \widetilde{B}^{s}_{\rho,q}(\Omega)} 
		:= \inf_{\substack{\widetilde{g}\in B^{s}_{\rho,q}(\R^d):\, g=\widetilde{g}|_{\Omega},\\ \supp(\widetilde{g})\subseteq \overline{\Omega}}} \norm{\widetilde{g} \sep B^{s}_{\rho,q}(\R^d)},
	\end{align*}
	cf.\ Triebel \cite[Def.\ 2.1]{T08}. Further from \cite[Thm.\ 3.30]{T08} and $\min\{\rho,q\}>1$ it follows that $\widetilde{B}^{s}_{\rho,q}(\Omega) = ( B^{-s}_{\rho',q'}(\Omega) )'$ such that for all $j=1,\ldots,d$ we can estimate
	\begin{align*}
		\abs{A(\nabla u_{\sigma,\theta,d})_j(\varphi)} \lesssim \norm{\varphi \sep B^{-s}_{\rho',q'}(\Omega)}, \qquad \varphi \in \D(\Omega).
	\end{align*}
	Therefore the representation formula \link{eq:representation_f} together with \autoref{lem:embedding}\link{item:lift} yield that for all $\psi\in\D(\Omega)$ there holds
	\begin{align*}
		\abs{f_{\sigma,\theta,d}^{[p]}(\psi)}
		&= \abs{\sum_{j=1}^d \int_\Omega A(\nabla u_{\sigma,\theta,d})_j(x)\, \frac{\partial\psi}{\partial x_j}(x)\d x} 
		\\
		&
		\leq \sum_{j=1}^d \abs{A(\nabla u_{\sigma,\theta,d})_j \left( \frac{\partial\psi}{\partial x_j} \right)}
		\\
		&\lesssim \sum_{j=1}^d \norm{\frac{\partial\psi}{\partial x_j} \sep B^{-s}_{\rho',q'}(\Omega)} 
		\\
		&
		\lesssim \norm{\psi \sep B^{1-s}_{\rho',q'}(\Omega)} 
		\\
		&
		\leq \norm{\psi \sep \widetilde{B}^{1-s}_{\rho',q'}(\Omega)}. 
	\end{align*}
	Now we again employ \cite[Thm.\ 3.30]{T08} to see that $\D(\Omega)$ is dense in $\widetilde{B}^{1-s}_{\rho',q'}(\Omega)$ and hence $f_{\sigma,\theta,d}^{[p]} \in \big( \widetilde{B}^{1-s}_{\rho',q'}(\Omega) \big)'= B^{-1+s}_{\rho,q}(\Omega)$, as claimed.
	
	It remains to prove \link{item:f_in_Lrho}. For this purpose, note that for $f_{\sigma,\theta,d}^{[p]}\in L_{\rho}(\Omega)=(L_{\rho'}(\Omega))'$ it is sufficient to find $f_d\in L_{\rho}(\Omega)$ such that 
	\begin{align}\label{eq:proof_f}
		f_{\sigma,\theta,d}^{[p]}(\psi) = \int_\Omega f_d(x)\,\psi(x)\d x, \qquad \psi\in\D(\Omega),
	\end{align} 
	since $\D(\Omega)$ is dense in $L_{\rho'}(\Omega)$ if $1\leq \rho' < \infty$. We claim that this $f_d$ is given by the restriction of $f_d=f\circ r_d$ to $\Omega$, where
	\begin{align*}
		f(r):=-v_{(p-1)\sigma,\theta}'(r) - v_{(p-1)\sigma,\theta}(r)\,\frac{d-1}{r} \qquad \text{for a.e.}\qquad r>0.
	\end{align*}
	Recall that due to \autoref{lem:prop_uv} $v_{(p-1)\sigma,\theta}$ is continuous with support in $[1/4,3/4]$. Hence, the function $r\mapsto g(r):=v_{(p-1)\sigma,\theta}(r)/r$ belongs to $L_{\rho}(\R_+)$. Moreover, it is clear that our assumption $w_{(p-1)\sigma,\theta}'\in L_{\rho}(\R)$ implies that also $v_{(p-1)\sigma,\theta}'\in L_{\rho}(\R_+)$. Therefore, we have $f\in L_{\rho}(\R_+)$ and from \autoref{lem:rotation_reg} we conclude $f_d\in L_{\rho}(\R^d)$. This shows that indeed $f_d\in L_{\rho}(\Omega)$. Thus, we are left with proving \link{eq:proof_f}. To this end, let $\psi\in\D(\Omega)$ be arbitrarily fixed and assume for a moment that $d\geq 2$.	
	Then $B_1(0)\subseteq \Omega$ and a transformation into polar coordinates $x=r\,\vartheta(\phi)$, $(r,\phi)\in [0,\infty)\times \Phi$, yields
	\begin{align}\label{eq:intF}
		\int_{\Omega} f_d(x) \, \psi(x)\d x
		&= \int_\Phi \int_{1/4}^{3/4} f(r)\, \psi(r\,\vartheta(\phi)) \, r^{d-1} \, T(\phi) \d r \d\phi
	\end{align}	
	(cf.\ the proof of \autoref{lem:rotation_reg}). 
	Since $v'_{(p-1)\sigma,\theta}$ belongs to $L_1(\R_+)$ and $\psi$ is smooth, we may use integration by parts to see that	
	\begin{align*}
		& \int_{1/4}^{3/4} v'_{(p-1)\sigma,\theta}(r) \, \psi(r\,\vartheta(\phi)) \, r^{d-1} \d r \\
		&\qquad= \big[ v_{(p-1)\sigma,\theta}(r)\,\psi(r\,\vartheta(\phi)) \, r^{d-1}\big]_{r=1/4}^{3/4} - \int_{1/4}^{3/4} v_{(p-1)\sigma,\theta}(r) \, \frac{\d }{\d r} \!\left(\psi(r\,\vartheta(\phi)) \, r^{d-1}\right)(r) \d r
	\end{align*}
	is finite, because the boundary term vanishes and
	\begin{align*}
		\frac{\d }{\d r} \!\left(\psi(r\,\vartheta(\phi)) \, r^{d-1}\right)(r)
		&=\frac{\d }{\d r} \left(\psi(r\,\vartheta(\phi))\right)(r)\, r^{d-1} + (d-1) \, \psi(r\,\vartheta(\phi))\, r^{d-2}\\
		&=\distr{\nabla\psi(r\,\vartheta(\phi))}{\vartheta(\phi)}_{\R^d}\, r^{d-1} +  \frac{d-1}{r}\,\psi(r\,\vartheta(\phi))\, r^{d-1}
	\end{align*}
	as well as $v_{(p-1)\sigma,\theta}$ are bounded on $[1/4,3/4]$.
	Hence, for the inner integral in \link{eq:intF} we find
	\begin{align*}
		&\int_{1/4}^{3/4} f(r) \, \psi(r\,\vartheta(\phi)) \, r^{d-1} \d r \\
		&\qquad = -\int_{1/4}^{3/4} v'_{(p-1)\sigma,\theta}(r) \, \psi(r\,\vartheta(\phi)) \, r^{d-1} \d r 
					- \int_{1/4}^{3/4} v_{(p-1)\sigma,\theta}(r)\,\frac{d-1}{r} \, \psi(r\,\vartheta(\phi)) \, r^{d-1} \d r \\
		&\qquad = \int_{1/4}^{3/4} v_{(p-1)\sigma,\theta}(r) \, \distr{ \nabla\psi(r\,\vartheta(\phi)) }{ \frac{r\,\vartheta(\phi)}{r} }_{\R^d} \, r^{d-1} \d r
	\end{align*}
	and thus \link{eq:representation_f} shows that we indeed have \link{eq:proof_f}:
	\begin{align}
		\int_{\Omega} f_d(x) \, \psi(x)\d x
			&= \int_\Phi \int_{1/4}^{3/4} f(r)\, \psi(r\,\vartheta(\phi)) \, r^{d-1} \, T(\phi) \d r \d\phi \nonumber\\
			&=\int_\Phi \int_{1/4}^{3/4} v_{(p-1)\sigma,\theta}(r) \, \distr{ \nabla\psi(r\,\vartheta(\phi)) }{ \frac{r\,\vartheta(\phi)}{r} }_{\R^d} \, r^{d-1} \, T(\phi) \d r \d\phi \nonumber\\
			&=\int_{\Omega} v_{(p-1)\sigma,\theta,d}(x) \, \distr{\nabla\psi(x)}{\frac{x}{\abs{x}_2}}_{\R^d} \d x  \label{eq:int_Fpsi}\\
		&=f_{\sigma,\theta,d}^{[p]}(\psi). \nonumber
	\end{align}
	Finally, a similar calculation shows that \link{eq:int_Fpsi} remains valid also for $d=1$. So, the proof is complete.
\end{proof}

Now we are well-prepared to give profound proofs of our main results stated in \autoref{thm:main} and \autoref{thm:mainL}.

\subsection{Proof of \autoref{thm:main}}\label{subsect:proof-main}
\begin{proof}
	Let $d\in\N$, as well as $2 \leq p<\infty$, and $0<\epsilon<1/p$ be given fixed. Further let $\Omega \subseteq \R^d$ be either $\R^d$ itself, a bounded Lipschitz domain, or an interval (if $d=1$). Since $\Omega$ is open it contains inner points. By a simple translation and dilation argument (see, e.g.\ \cite[Sect.\ 4]{DahDieHar+2016}) we may w.l.o.g.\ assume that the Euclidean ball of radius one, $B_1(0)$, is contained in $\Omega$. In what follows we will choose specific values $\sigma\in(0,1)$ as well as $\theta\in(1,\infty)$ and define $u:=u_{\sigma,\theta,d}$ and $f:=f_{\sigma,\theta,d}^{[p]}$ according to \link{eq:def_uf}. From \autoref{lem:p-poisson} it then follows that $u\in W^1_{p,0}(\Omega)$ is a weak solution to \link{eq:strong_ppoisson} with right-hand side $f\in W^{-1}_{p'}(\Omega)$ and that the supports of $u$ and $f$ are contained in $B_{4/5}(0)$.
	
	Given $1<\mu\leq\infty$ and $\epsilon\,(p-1)<\lambda < 1-\epsilon$ we choose $\theta$ such that $0<1-\epsilon < 1/\theta < 1$ and define
	\begin{align*}
		\sigma:= \frac{\lambda}{p-1}-\frac{1-1/\theta}{(p-1)\,\mu}.
	\end{align*}
	Then it is easily seen that
	\begin{align}\label{eq:cond_sigma}
		0< \frac{\lambda}{p-1}-\epsilon < \sigma \leq (p-1)\,\sigma < \frac{1}{\theta} < 1.
	\end{align}	
	Indeed, the lower bound on $\lambda$ shows that $\lambda/(p-1)-\epsilon$ is strictly positive. If $1<\mu<\infty$, we note that $p\geq 2$ implies $1-\epsilon\,(p-1)\,\mu < 1-\epsilon < 1/\theta$ and hence
	\begin{align*}
		\frac{\lambda}{p-1} -\epsilon
		= \frac{\lambda}{p-1} + \frac{1-\epsilon\,(p-1)\,\mu-1}{(p-1)\,\mu}
		< \frac{\lambda}{p-1}+\frac{1/\theta-1}{(p-1)\,\mu}
		= \sigma,
	\end{align*}	
	while the corresponding estimate for $\mu=\infty$ is trivial since $\epsilon>0$.
	Moreover, $1/\theta < 1$ yields
	\begin{align*}
		\sigma \leq (p-1)\,\sigma \leq \lambda < 1-\epsilon < \frac{1}{\theta} < 1
	\end{align*}
	which completes the proof of \link{eq:cond_sigma}.	
	
	Next we note that \link{eq:cond_sigma} particularly implies that 
	\begin{align*}
		\min\left\{ \theta\,(1+\sigma), \frac{1-\sigma}{1-1/\theta}\right\} > 1.
	\end{align*}
	So, we can employ \autoref{lem:Regw}\link{item:w_in_B} to see that for $1 \leq \rho \leq \infty$ there holds 
	\begin{align}\label{eq:proof_winB}
		w_{\sigma,\theta} \in B^{s}_{\rho,q}(\R)
		\quad\text{ if and only if }\quad
		s=s_\rho \;\text{ and }\; q=\infty,
		\quad\text{or}\quad 
		s<s_\rho \;\text{ and }\; 0<q \leq \infty,
	\end{align}
	where 
	\begin{align*}
		s_\rho
		:=\sigma + \frac{1-1/\theta}{\rho} 
		= \frac{\lambda}{p-1} + \left( \frac{1}{\rho} - \frac{1}{(p-1)\,\mu} \right) \left(1-\frac{1}{\theta}\right).
	\end{align*}
	Note that our assumptions imply that $0<1-1/\theta<\epsilon$, 
	\begin{align*}
		\frac{1}{\rho} - \frac{1}{(p-1)\,\mu} \in \left\{\begin{array}{lcrl}
			(-1,0) 
				&\quad\text{if}&  (p-1)\,\mu < \rho \!\!\!&\leq \infty, \\
			\{0\}  
				&\quad\text{if}&  \rho \!\!\!&=(p-1)\,\mu, \\
			(0,1)  
				&\quad\text{if}&  1 < \rho \!\!\!&< (p-1)\,\mu, 
		\end{array}\right. 
	\end{align*}
	as well as $0< \lambda/(p-1) \pm \epsilon<1$. If $\rho$ is large, we have $\lambda/(p-1)-\epsilon < s_\rho < \lambda/(p-1)$. Thus \link{eq:proof_winB} and \autoref{lem:reg_u}\link{item:reg_ud} prove 
	\begin{align*}
		u=u_{\sigma,\theta,d} \in B^{1+\frac{\lambda}{p-1}-\epsilon}_{\rho,q}(\Omega) \setminus B^{1+\frac{\lambda}{p-1}}_{\rho,q}(\Omega), \qquad 0<q\leq\infty,
	\end{align*}	
	for this case. The regularity statements for $u$ in the remaining cases are obtained likewise.
	
	Similarly, \link{eq:cond_sigma} and \autoref{lem:Regw}\link{item:w_in_B} show that for $1 \leq \rho \leq \infty$ there holds 
	\begin{align*}
		w_{(p-1)\sigma,\theta} \in B^{s}_{\rho,q}(\R)
		\quad\text{ if and only if }\quad
		s=\widetilde{s}_\rho \;\text{ and }\; q=\infty,
		\quad\text{or}\quad 
		s<\widetilde{s}_\rho \;\text{ and }\; 0<q \leq \infty,
	\end{align*}
	where now (depending on the relation of $\rho$ and $\mu$ to each other)
	\begin{align*}
		\widetilde{s}_\rho
		:=(p-1)\,\sigma + \frac{1-1/\theta}{\rho} 
		= \lambda + \left( \frac{1}{\rho} - \frac{1}{\mu} \right) \left(1-\frac{1}{\theta}\right) \in(\lambda - \epsilon, \lambda + \epsilon)\subsetneq (0,1).
	\end{align*}
	Therefore, we can use \autoref{lem:reg_u}\link{item:reg_A} to deduce the regularity statements for $A(\nabla u)$. In particular 
	we have 
	$A(\nabla u) \in \big( B_{\mu,\infty}^\lambda(\Omega) \big)^d$ such that by \autoref{lem:reg_u}\link{item:f_in_B} $f=f_{\sigma,\theta,d}^{[p]} \in B^{-1+\lambda}_{\mu,\infty}(\Omega)$,
	as claimed.
\end{proof}

\subsection{Proof of \autoref{thm:mainL}}\label{subsect:proof-mainL}

In order to show \autoref{thm:mainL} we essentially follow the lines of the proof of \autoref{thm:main}. So let us focus on the necessary modifications only.
\begin{proof}
	Given $1<\mu\leq \infty$ and $0<\epsilon<\min\{1/(p-1), 1-1/(p-1)\}$ (note that this time $p>2$!) we choose $\theta$ such that $0<1-\epsilon < 1/\theta <1$ and define
	\begin{align*}
		\sigma:= \frac{1}{p-1}-\frac{1-1/\theta}{(p-1)\,\mu}.
	\end{align*}
	Then there holds
	\begin{align*}
		0<\sigma < \frac{1}{\theta} < (p-1)\,\sigma \leq 1.
	\end{align*}
	Indeed, $\mu>1$ and $0 < 1/\theta < 1$ show that $0 \leq (1-1/\theta)/\mu < 1-1/\theta < 1$. This proves $0<\sigma$ as well as $1/\theta < (p-1)\,\sigma \leq 1$. Hence, we also have $\sigma \leq 1/(p-1) < 1-\epsilon < 1 /\theta$ due to our assumption on $\epsilon$.
	
	Now the claimed regularity of $u$ follows exactly as in the proof of \autoref{thm:main}, where this time our restrictions on $\epsilon$ ensure that $0<1/(p-1)\pm \epsilon < 1$.
	
	In order to prove the regularity statement for $A(\nabla u)$ we like to apply \autoref{lem:reg_u}\link{item:reg_A}. To this end, we have to show that $w_{(p-1)\sigma,\theta}\in W^1_\rho(\R)$ for $1<\rho<\infty$ if and only if $\rho<\mu$. 
	By \autoref{lem:Regw}\link{item:w_in_L_rho} this reduces to the claim $w'_{(p-1)\sigma,\theta}\in L_\rho(\R)$.
	If $\mu=\infty$, we actually have $(p-1)\,\sigma = 1$. Therefore, from \autoref{lem:Regw}\link{item:w'_in_L_rho} it follows that $w'_{(p-1)\sigma,\theta}\in L_\rho(\R)$ for all $0<\rho \leq \infty$.
	 On the other hand, if $\mu<\infty$, then $0<(p-1)\,\sigma < 1$. Hence, in this case we have $w'_{(p-1)\sigma,\theta}\in L_\rho(\R)$ if and only if
	\begin{align*}
		\frac{1-(p-1)\,\sigma}{1-1/\theta} = \frac{1}{\mu} < \frac{1}{\rho}.
	\end{align*}
	In conclusion, $A(\nabla u) \in \big( W^1_\rho(\Omega) \big)^d$ for $1<\rho<\infty$ is equivalent to $\rho < \mu$, as claimed.

	It remains to prove that $f=f_{\sigma,\theta,d}^{[p]}$ belongs to $L_\nu(\Omega)$. If $\nu>1$, this follows from \autoref{lem:reg_u}\link{item:f_in_Lrho} and the calculations above. However, in view of the compact support of $f$, this lower bound on $\nu$ can be dropped.
\end{proof}


\addcontentsline{toc}{chapter}{References}
\begin{thebibliography}{10}
	\small\itemsep-2pt

	\bibitem{BalDieWei2018}
	A.~K. Balci, L.~Diening, and M.~Weimar.
	\newblock Higher order {C}alder{\'o}n-{Z}ygmund estimates for the $p$-{L}aplace
	equation.
	\newblock Submitted for publication, 2019.
	\newblock Available as \emph{ruhr.paD Preprint} No. 2019-02, arXiv:1904.03388.
	
	\bibitem{CDD01}
	A.~Cohen, W.~Dahmen, and R.~A. DeVore.
	\newblock Adaptive wavelet methods for elliptic operator equations:
	{C}onvergence rates.
	\newblock {\em Math. Comp.}, 70:\penalty0 27--75, 2001.
	
	\bibitem{DahDieHar+2016}
	S.~Dahlke, L.~Diening, C.~Hartmann, B.~Scharf, and M.~Weimar.
	\newblock Besov regularity of solutions to the $p$-{P}oisson equation.
	\newblock {\em Nonlinear Anal.}, 130:\penalty0 298--329, 2016.
	
	\bibitem{DeT1982}
	F.~De~Thelin.
	\newblock Local regularity properties for the solutions of a nonlinear partial
	differential equation.
	\newblock {\em Nonlinear Anal.}, 6\penalty0 (8):\penalty0 839--844, 1982.
	
	\bibitem{GasMor2014}
	F.~D. Gaspoz and P.~Morin.
	\newblock Approximation classes for adaptive higher order finite element
	approximation.
	\newblock {\em Math. Comp.}, 83\penalty0 (289):\penalty0 2127--2160, 2014.
	
	\bibitem{HarWei2018}
	C.~Hartmann and M.~Weimar.
	\newblock Besov regularity of solutions to the $p$-{P}oisson equation in the
	vicinity of a vertex of a polygonal domain.
	\newblock {\em Results Math.}, 73:\penalty0 41, 2018.
	
	\bibitem{KMM07}
	N.~Kalton, S.~Mayboroda, and M.~Mitrea.
	\newblock Interpolation of {H}ardy-{S}obolev-{B}esov-{T}riebel-{L}izorkin
	spaces and applications to problems in partial differential equations.
	\newblock In L.~{De Carli} and M.~Milman, editors, {\em Interpolation {T}heory
		and {A}pplications ({C}ontemporary {M}athematics 445)}, pages 121--177. Amer.
	Math. Soc., Providence, RI, 2007.
	
	\bibitem{Lin2006}
	P.~Lindqvist.
	\newblock {\em Notes on the $p$-Laplace equation}.
	\newblock University of Jyv\"askyl\"a, Department of Mathematics and
	Statistics, 2006.
	
	\bibitem{Ryc1999}
	V.~S. Rychkov.
	\newblock On restrictions and extensions of the {B}esov and
	{T}riebel-{L}izorkin spaces with respect to {L}ipschitz domains.
	\newblock {\em J.~London Math.\ Soc.\ (2)}, 60\penalty0 (1):\penalty0 237--257,
	1999.
	
	\bibitem{Sav1998}
	G.~Savar{\'e}.
	\newblock Regularity results for elliptic equations in {L}ipschitz domains.
	\newblock {\em J. Funct. Anal.}, 152:\penalty0 176--201, 1998.
	
	\bibitem{SicSkrVyb2012}
	W.~Sickel, L.~Skrzypczak, and J.~Vybiral.
	\newblock On the interplay of regularity and decay in case of radial functions
	{I}. {I}nhomogeneous spaces.
	\newblock {\em Commun. Contemp. Math.}, 14\penalty0 (1):\penalty0
	1250005--1--60, 2012.
	
	\bibitem{Sim1978}
	J.~Simon.
	\newblock R\'egularit\'e de la solution d'une \'equation non lin\'eaire dans
	{$\mathbb{R}^{N}$}.
	\newblock In P.~B\'enilan and J.~Robert, editors, {\em Journ\'ees d'{A}nalyse
		{N}on {L}in\'eaire ({P}roc. {C}onf., {B}esan\c con, 1977)}, volume 665 of
	{\em Lecture Notes in Math.}, pages 205--227. Springer, Berlin, 1978.
	
	\bibitem{Sim1979}
	J.~Simon.
	\newblock Regularit\'e de la compos\'ee de deux fonctions et applications.
	\newblock {\em Boll. Un. Mat. Ital. B (5)}, 16\penalty0 (2):\penalty0 501--522,
	1979.
	
	\bibitem{T83}
	H.~Triebel.
	\newblock {\em Theory of {F}unction {S}paces}.
	\newblock Birkh\"auser, Basel/Boston/Stuttgart, 1983.
	
	\bibitem{T92}
	H.~Triebel.
	\newblock {\em Theory of {F}unction {S}paces {II}}, volume~84 of {\em
		Monographs in Mathematics}.
	\newblock Birkh\"auser Verlag, Basel, 1992.
	
	\bibitem{T06}
	H.~Triebel.
	\newblock {\em Theory of {F}unction {S}paces III}.
	\newblock Birkh\"auser, Basel, 2006.
	
	\bibitem{T08}
	H.~Triebel.
	\newblock {\em Function {S}paces and {W}avelets on {D}omains}, volume~7 of {\em
		EMS Tracts in Mathematics}.
	\newblock European Mathematical Society (EMS), Z\"urich, 2008.
	
	\bibitem{Wei2016}
	M.~Weimar.
	\newblock Almost diagonal matrices and {B}esov-type spaces based on wavelet
	expansions.
	\newblock {\em J.\ Fourier Anal.\ Appl.}, 22\penalty0 (2):\penalty0 251--284,
	2016.
	
\end{thebibliography}
\end{document}